\documentclass[12pt,reqno]{amsart}
\usepackage{textcomp}
\usepackage{amsmath}
\usepackage{mathrsfs}
\usepackage{longtable} 
\usepackage{array}
\usepackage{leftindex}
\usepackage{amsthm}
\usepackage[linktocpage]{hyperref} 
\usepackage{amsfonts,graphics,amsthm,amsfonts,amscd,latexsym}
\usepackage{epsfig}
\usepackage{flafter}
\usepackage{mathtools}
\usepackage{comment} 
\usepackage{stmaryrd}
\usepackage{enumitem}
\usepackage{epigraph}
\usepackage{tikz-cd}

\hypersetup{
colorlinks=true,
linkcolor=red,
citecolor=green,
filecolor=blue,
urlcolor=blue
}

\usepackage{tikz}
\usetikzlibrary{graphs,positioning,arrows,shapes.misc,decorations.pathmorphing}

\tikzset{
>=stealth,
every picture/.style={thick},
graphs/every graph/.style={empty nodes},
}

\tikzstyle{vertex}=[
draw,
circle,
fill=black,
inner sep=1pt,
minimum width=5pt,
]

\usetikzlibrary{decorations.pathmorphing}
\tikzstyle{printersafe}=[decoration={snake,amplitude=0pt}]

\usepackage[position=top]{subfig}
\usepackage[alphabetic]{amsrefs}
\usepackage{amssymb}
\usepackage{color}

\calclayout

\newcommand{\Pic}{\operatorname{Pic}}

\newcommand{\mult}{\operatorname{mult}}

\newcommand{\oo}{\mathcal{O}}

\def\O#1.{\mathcal {O}_{#1}}
\def\pr #1.{\mathbb P^{#1}}
\def\af #1.{\mathbb A^{#1}}
\def\ses#1.#2.#3.{0\to #1\to #2\to #3 \to 0}
\def\xrar#1.{\xrightarrow{#1}}
\def\K#1.{K_{#1}}
\def\bA#1.{\mathbf{A}_{#1}}
\def\bM#1.{\mathbf{M}_{#1}}
\def\bL#1.{\mathbf{L}_{#1}}
\def\bB#1.{\mathbf{B}_{#1}}
\def\bK#1.{\mathbf{K}_{#1}}
\def\subs#1.{_{#1}}
\def\sups#1.{^{#1}}

\newcommand{\PP}{\mathbb{P}}
\newcommand{\ZZ}{\mathbb{Z}}
\newcommand{\QQ}{\mathbb{Q}}
\newcommand{\CC}{\mathbb{C}}

\newcommand{\Aut}{\operatorname{Aut}}

\DeclareMathOperator{\lct}{lct}

\def \ge {\geqslant}
\def \le {\leqslant}
\def \geq {\geqslant}
\def \leq {\leqslant}

\usepackage{tikz}
\usetikzlibrary{matrix,arrows,decorations.pathmorphing}

\newtheorem{theorem}{Theorem}[section]
\newtheorem{lemma}[theorem]{Lemma}
\newtheorem{proposition}[theorem]{Proposition}
\newtheorem{corollary}[theorem]{Corollary}

\newtheorem{conjecture}[theorem]{Conjecture}

\theoremstyle{definition}

\newtheorem{assumptions}[theorem]{Assumptions}

\newtheorem{definition}[theorem]{Definition}
\newtheorem{example}[theorem]{Example}

\newtheorem{question}[theorem]{Question}
\newtheorem{construction}[theorem]{Construction}

\newtheorem{remark}[theorem]{Remark}

\theoremstyle{remark}

\numberwithin{equation}{section}

\usepackage[all]{xy}

\newcounter{rownumber}[figure]
\newcounter{rownumber-irr}[figure]
\newcounter{rownumber-p1}[figure]
\begin{document}

\title{Finiteness of projective pluricanonical representation for automorphisms of complex manifolds}

\begin{abstract}
We study the action of the group of bimeromorphic automorphisms~\mbox{$\mathrm{Bim}(X)$} of a compact complex manifold $X$ on the image of the pluricanonical map, which we call the projective pluricanonical representation of this group. 
If $X$ is a Moishezon variety, then the image of $\mathrm{Bim}(X)$ via such a representation is a finite group by a classical result due to Deligne and Ueno. 
We prove that this image is a finite group under the assumption that for the Kodaira dimension $\kappa(X)$ of $X$ we have~\mbox{$\kappa(X)=\dim X-1$}. To this aim, we prove a version of the canonical bundle formula in relative dimension $1$ which works for a proper morphism from  a complex variety to a projective variety. 
In particular, this establishes the analytic version of Prokhorov--Shokurov conjecture in relative dimension~$1$. 
Also, we observe that the analytic version of this conjecture does not hold in relative dimension~$2$. 
\end{abstract}

\author[K.~Loginov]{Konstantin Loginov}
\address{Steklov Mathematical Institute of Russian Academy of Sciences, 
8 Gubkina street, Moscow 119991, Russia.  
\newline
Laboratory of Algebraic Geometry, National Research University Higher School of Economics, Russian Federation.
\newline
Centre of Pure Mathematics, MIPT, Moscow, Russia.}
\email{loginov@mi-ras.ru}

\author[C.~Shramov]{Constantin Shramov}
\address{Steklov Mathematical Institute of Russian Academy of Sciences, 
8 Gubkina street, Moscow 119991, Russia.  
\newline 
HSE University, Russian Federation,
Laboratory of Algebraic Geometry, 6 Usacheva str., Moscow, 119048, Russia.}
\email{costya.shramov@gmail.com}

\maketitle
\setcounter{tocdepth}{1} 
\tableofcontents
\section{Introduction}
In this paper, we study groups of bimeromorphic automorphisms of complex manifolds. More precisely, we study the image of the pluricanonical representation of the group of bimeromorphic self-maps. 

In what follows, we denote by $\mathrm{Aut}(X)$ the group of automorphisms of a compact complex manifold $X$, by $\mathrm{Bim}(X)$ the group of bimeromorphic self-maps of $X$, and by~$\kappa(X)$ the Kodaira dimension of~$X$. 
Let us denote by $\rho$ the \emph{pluricanonical representation} of the group $\mathrm{Bim}(X)$ which comes from the action of $\mathrm{Bim}(X)$ on the space of holomorphic $m$-forms. 
Also, by $\overline{\rho}$ we denote the projectivization of the pluricanonical representation, that is, the homomorphism of $\mathrm{Bim}(X)$ to the autormorphism group of the projectivization of the space of holomorphic $m$-forms. 
Thus, $\overline{\rho}(\mathrm{Bim}(X))$ acts on the image of the pluricanonical map, see Section \ref{subsec-pluricanonical-rep} for details. For convenience, we call it the \emph{projective pluricanonical representation}. 
Recall the following classical

\begin{theorem}[{\cite[Theorem 14.10]{Ue75}}]
\label{thm-finite-moishezon}
Let $X$ be a compact complex manifold. Assume that~$X$ is Moishezon. Then the group $\rho(\mathrm{Bim}(X))$, and hence also 
the group~\mbox{$\overline{\rho}(\mathrm{Bim}(X))$}, is finite.
\end{theorem}

There exist analogous results for projective pairs, see \cite{FG14}, \cite{HX16}, and 
also~\cite{FX25}. 
Moreover, for arbitrary compact complex manifolds a weaker result is known. 
Namely, we say that a group has \emph{bounded finite subgroups}, if the order of any finite subgroup of such group is bounded by a constant which is independent of the subgroup.

\begin{theorem}[{\cite[Theorem~1.4]{Loginov}}]
\label{prop-pluri-bfs}
Let $X$ be a compact complex manifold. Then both the image of the pluricanonical representation $\rho(\mathrm{Bim}(X))$ and the image of the projective pluricanonical representation $\overline{\rho}(\mathrm{Bim}(X))$ have bounded finite subgroups. 
\end{theorem}

In contrast to Theorem \ref{thm-finite-moishezon}, in Theorem \ref{prop-pluri-bfs} we do not need the assumption that $X$ is Moishezon, but instead of finiteness, we establish a weaker property of $\rho(\mathrm{Bim}(X))$. However, this is the best we can get as there exist compact complex manifolds with infinite~\mbox{$\rho(\mathrm{Bim}(X))$}, see Example~\ref{example-infinite} below. 
Thus, the best one can hope for in general is finiteness of the group~\mbox{$\overline{\rho}(\mathrm{Bim}(X))$}. 

If $X$ is a curve, then finiteness of ${\rho}(\mathrm{Bim}(X))$ and $\overline{\rho}(\mathrm{Bim}(X))$ follows either from Theorem \ref{thm-finite-moishezon} or from a simple case by case analysis. 
If $X$ is a compact complex surface, then $\overline{\rho}(\mathrm{Bim}(X))$ is finite. Indeed, 
finiteness of $\overline{\rho}(\mathrm{Bim}(X))$ is trivial for~\mbox{$\kappa(X)=0$}, while 
for $\kappa(X)=1$ it is proved in \cite[Proposition~1.2]{PSh20}. If~\mbox{$\kappa(X)=\dim X$}, then the group $\mathrm{Bim}(X)$ is itself finite by \cite[Corollary~14.3]{Ue75}.  
In higher dimensions, we do not know about finiteness results 
for $\overline{\rho}(\mathrm{Bim}(X))$ when the Kodaira dimension is not maximal possible.  
Hence we formulate the following

\begin{question}
\label{main-question}
Does there exist a compact complex manifold $X$ such that the image of the projective pluricanonical representation $\overline{\rho}(\mathrm{Bim}(X))$ is infinite?
\end{question}

In this paper, we prove 
\begin{theorem}
\label{main-theorem}
Let $X$ be a compact complex manifold of Kodaira dimension 
\[
\kappa(X)=\dim X-1.
\] 
Then  $\overline{\rho}(\mathrm{Bim}(X))$ is finite. 
\end{theorem}

\begin{remark}
We point out that an analog of Theorem~\ref{main-theorem} does not hold 
if we replace the projective pluricanonical representation of the (bimeromorphic) automorphism 
group by the image under the algebraic reduction, even in the case of compact complex surfaces,
see e.g.~\cite{Shramov-Elliptic}.
\end{remark}

We also formulate a similar question for varieties defined over a field of positive characteristic. 

\begin{question}
\label{second-question}
Does there exist a projective variety $X$ defined over a field of positive characteristic such that the image of the projective pluricanonical representation~\mbox{$\overline{\rho}(\mathrm{Bim}(X))$}  is infinite?
\end{question}

The negative answer to Question \ref{second-question} is known only in the case of curves, varieties of general type, and elliptic surfaces of Kodaira dimension $1$, see~
\mbox{\cite[Corollary 1.5]{Gu20}}. In particular, the answer is not known in the case of quasi-elliptic surfaces of Kodaira dimension $1$. 

To prove Theorem \ref{main-theorem}, we establish a version of the canonical bundle formula which works for a proper morphism from a complex variety to a projective variety. 
This proves a version of Prokhorov--Shokurov conjecture in the analytic setting in relative dimension at most $1$, see Conjecture \ref{conj-analytic-PS} and Proposition \ref{prop-proof-of-conjecture}. 
Recall that the canonical bundle formula provides a way to relate the canonical class of the total space of the fibration with the canonical class of its base. It works under some natural assumptions on the fibration. We refer the reader to Section \ref{sec-preliminaries} for the relevant notation.

\begin{proposition}[{see Proposition \ref{prop:canonical-bundle-formula}, cf. \cite[2.15]{Fujita}}]
\label{prop:canonical-bundle-formula-intro}
Let $X$ be a compact complex manifold  with a 
fibration $\sigma\colon X\to Y$ over a smooth projective variety $Y$ such that the  typical fiber of $\sigma$ is an elliptic curve. Let $R$ be a $\sigma$-vertical $\QQ$-divisor on $X$. Assume further that  $K_X+R\sim_{\mathbb{Q}}0/Y$. 
Then
\begin{equation}
\label{eq-MD-intro}
K_X+R\sim_{\QQ} \sigma^*\left(K_Y+\Delta+M \right),
\end{equation}
where $M$ is a $\mathbb{Q}$-divisor on $Y$, and 
$$
\Delta=\sum\limits_{Z\subset Y} (1-\lct(X,R+\sigma^*Z)) Z.
$$
Here $Z$ runs through the set of all prime divisors on $Y$, and the log-canonical threshold is computed over a typical point of $Z$. 
If the $j$-invariant map $j\colon Y\dashrightarrow \mathbb{P}^1$ is a morphism, then 
$M\sim_{\QQ}\frac{1}{12}j^*Q$  
for a point $Q\in \mathbb{P}^1$, in particular, $M$ is nef. 
Moreover, if $(X, R)$ is an lc sub-pair, then $(Y, \Delta)$ is an lc sub-pair. 
\end{proposition}
The $\mathbb{Q}$-divisor $\Delta$ in \eqref{eq-MD-intro} is called the \emph{discriminant $\mathbb{Q}$-divisor} of $\sigma$, and $M$ is called the \emph{moduli $\mathbb{Q}$-divisor} of $\sigma$. 
This proposition is well known in the case when $X$ is a projective variety, see \cite[Example 7.16]{PS09}. 
To prove Proposition \ref{prop:canonical-bundle-formula-intro}, we reduce it to the two-dimensional case, where it is proved in the same way as for projective surfaces, see e.g. \cite[\S 2]{Fujita}. More results on the canonical bundle formula can be found in Section \ref{sec-canonical-bundle}, where we also note that Conjecture \ref{conj-analytic-PS} does not hold for compact complex manifolds in the relative dimension $2$, see Example \ref{exam-atiyah}.

We point out that many of our arguments used to prove Proposition \ref{prop:canonical-bundle-formula-intro} go in parallel with those of Fujita in \cite{Fujita}. Moreover, in \cite[2.17]{Fujita} the author expresses an expectation that most of his arguments should work in the analytic setting. In fact, our results confirm this expectation in the case when the morphism $\sigma\colon X\to Y$ is not necessarily projective, while $Y$ is assumed to be projective.

\subsection*{Sketch of proof}
We explain the idea of the proof of Theorem \ref{main-theorem}. Let $\sigma\colon X\dashrightarrow Y$ be the pluricanonical map, and $\Gamma = \overline{\rho}(\mathrm{Bim}(X))$ be the image of the projective pluricanonical representation. 
Note that $\Gamma$ acts on $Y$ biregularly, and the map $\sigma$ is $\mathrm{Bim}(X)$-equivariant. 

Let us first consider an ideal situation, and then we explain how to treat the general case. 
Assume that the pluricanonical map $\sigma\colon X\to Y$ is a morphism, $Y$ is smooth, a typical fiber of $\sigma$ is an elliptic curve, and that 
the $j$-invariant map~\mbox{$j\colon Y\dashrightarrow \mathbb{P}^1$} is a morphism. 
Using Proposition~\ref{prop:canonical-bundle-formula-intro}, 
we can write the canonical bundle formula
\[
K_X\sim_{\QQ} \sigma^*(K_Y+\Delta+M),
\]
where $\Delta$ is the discriminant $\mathbb{Q}$-divisor, and $M$ is the moduli $\mathbb{Q}$-divisor.
Using the fact that the map $\sigma$ is $\mathrm{Bim}(X)$-equivariant, in Proposition~\ref{prop-delta-prime-is-invariant} we show that $\Delta$ is $\Gamma$-invariant. In Lemma \ref{prop-action-on-target-of-j-map}, we explain that $M$ can also be chosen $\Gamma$-invariant. 
Also, since $X$ is smooth, the pair $(Y, \Delta+M)$ is lc. 
Since $\sigma$ is the pluricanonical map, we have $K_X\sim_{\QQ} \sigma^*H$ for some ample $\QQ$-divisor on $Y$. Hence, $K_Y+M+\Delta$ is an ample $\QQ$-divisor on $Y$. 
In Proposition \ref{prop-delta-prime-is-invariant} we prove that the discriminant $\mathbb{Q}$-divisor $\Delta$, as well as the moduli $\mathbb{Q}$-divisor $M$, is $\Gamma$-invariant. 
Then Proposition \ref{prop-group-preserves-pair} implies that in this case the automorphism group of the pair $(Y, \Delta+M)$ is finite, which means that $\Gamma$ is finite as well.

Now consider the general case.  
In this case, $\sigma\colon X\dashrightarrow Y$ need not be everywhere defined, $Y$ need not be smooth and the $j$-invariant map need not be everywhere defined as well.   
In Proposition \ref{prop-first-modification}, we construct bimeromorphic modifications $X_1$ of $X$ and $Y_1$ of $Y$, obtaining a $\mathrm{Bim}(X)$-equivariant fibration $\sigma_1\colon X_1\to Y_1$ of complex manifolds. 
Moreover, $\Gamma$ acts on $Y_1$ biregularly. 
In Proposition \ref{prop-elliptic-or-moishezon} we show that either its typical fiber is an elliptic curve, or $X$ is Moishezon, in which case the result follows from Theorem~\ref{thm-finite-moishezon}. Then we pass to a further modification $\sigma_2\colon X_2\to Y_2$ which enjoys the same properties as $\sigma_1$, and moreover, the $j$-invariant 
map~\mbox{$j\colon Y_2\to \mathbb{P}^1$} is a morphism. At this point the canonical class $K_{X_2}$ may not be $\QQ$-linearly trivial over $Y_2$. However, there exists a $\sigma_2$-vertical $\QQ$-divisor $R_2$ on $X_2$ such that $K_{X_2}+R_2$ is $\QQ$-linearly trivial over $Y_2$. 

%We construct another modification $\sigma_3\colon X_3\to Y_3$ with some special properties.  

Now, using Proposition \ref{prop:canonical-bundle-formula-intro}, we can write the canonical bundle formula
\[
K_{X_2}+R_2\sim_{\QQ} \sigma_2^*(K_{Y_2}+\Delta_{Y_2}+M_{Y_2}),
\]
where $\Delta_{Y_2}$ is the discriminant $\mathbb{Q}$-divisor, and $M_{Y_2}$ is the moduli $\mathbb{Q}$-divisor. By perturbing $R$ by a divisor $\sigma_2^*D$, where $D$ is a $\mathbb{Q}$-divisor on $Y_2$, in Proposition \ref{prop-delta-prime-is-invariant} we prove that the discriminant $\mathbb{Q}$-divisor $\Delta_{Y_2}=\Delta(\sigma_2, R_2)$, as well as the moduli $\mathbb{Q}$-divisor $M_{Y_2}$, is $\Gamma$-invariant. 
It turns out that the $\mathbb{Q}$-divisor $K_{Y_2}+\Delta_{Y_2}+M_{Y_2}$ is big and nef. Then Proposition \ref{prop-group-preserves-pair} implies that the automorphism group $(Y_2,\Delta_{Y_2}+M_{Y_2})$ is finite. Hence, $\Gamma$ is finite as well.

\textbf{Acknowledgements.} 
The authors thank Andrey Soldatenkov for useful discussions, and 
Kento Fujita for providing us with a reference to~\cite{Wilson}. 
This work was performed at the Steklov International Mathematical Center and supported by the Ministry of Science and Higher Education of the Russian Federation (agreement no. 075-15-2022-265), supported by
the HSE University Basic Research Program, and the Simons Foundation. The work of the first author is supported by the state assignment of MIPT (project FSMG-2023-0013). The first author is a Young Russian Mathematics award winner and would like to thank its sponsors and jury.

\section{Preliminaries}
\label{sec-preliminaries}
We work over the field of complex numbers. We refer the reader to \cite{Ue75} for the basic facts and definitions concerning complex varieties and manifolds. 
We refer to~\cite{KM98} for the basic definitions concerning pairs and their singularities.

\subsection{Complex varieties}
By a \emph{complex variety} we mean an irreducible reduced complex space. 
%A morphism of complex varieties is a holomorphic map between them.
 A smooth complex variety is called a \emph{complex manifold}. A \emph{complex surface} is a complex manifold of dimension~$2$. By a (Zariski) open subset of a complex variety $X$ we mean a subset of the form~$X \setminus Z$, where $Z$ is a closed analytic subset in $X$. By a \emph{typical point} of $X$ we mean a point in some non-empty Zariski open subset of $X$.
A \emph{typical fiber} of a holomorphic map $f\colon X\to Y$ is a fiber over a typical point
of~$Y$. 

A proper surjective holomorphic map $f\colon X\to Y$ of reduced complex spaces is called a \emph{bimeromorphic modification} if there exist  nowhere dense closed analytic subsets $V\subset X$ and $W\subset Y$ %that are  in every component of $X$ and $Y$ respectively, 
such that $f$ restricts to a biholomorphic map~\mbox{$X\setminus V\to Y\setminus W$}. A \emph{meromorphic map} $f\colon X\dashrightarrow Y$ of reduced complex analytic spaces $X$ and $Y$ is a holomorphic map defined outside a nowhere dense subset such that the closure of its graph~$\overline{\Gamma_f}\subset X\times Y$ is a closed analytic subset of $X\times Y$, and the natural projection~$\overline{\Gamma_f} \to X$ is a modification. The map $f$ is called \emph{bimeromorphic} if the natural projection $\overline{\Gamma_f}\to Y$ is a modification as well. 

Given a compact complex variety $X$, by $\mathrm{Bim}(X)$ we denote its group of bimeromorphic self-maps, and by $\mathrm{Aut}(X)$ we denote its group of biholomorphic self-maps.

\subsection{Divisors}
By a divisor on a complex variety we mean a Weil divisor, and by a $\mathbb{Q}$-divisor we mean a $\mathbb{Q}$-Weil $\mathbb{Q}$-divisor. 
We write $D\leq D'$ if $D=\sum a_i D_i$ and~\mbox{$D'=\sum a'_i D'_i$}, where $D_i$ and $D'_i$ are prime divisors, and $a_i\leq a'_i$ for any $i$. 
We say that a $\mathbb{Q}$-divisor is anti-effective if $-D\geq0$, that is, if $-D$ is effective. 
%For a $\mathbb{Q}$-divisor $D=\sum a_i D_i$, we put $D^{\geq 0} = \sum \max (0, a_i) D_i$, and $\lceil D\rceil = \sum \lceil a_i\rceil D_i$. 

We use the following notation. For two $\mathbb{Q}$-divisors $D$ and $D'$ on a complex variety~$X$, we write 
\[
D\sim_N D'
\]
for some positive integer $N$, 
if
$ND$ and $ND'$ are integral 
divisors and $ND\sim ND'$. 
Note that $\sim_N$ is an equivalence relation on the set of $\mathbb{Q}$-divisors. 
Clearly, $D \sim_N D'$ implies $D\sim_{M} D'$ for any $M$ divisible by $N$. 
We write
\[
D\sim_{\mathbb{Q}} D',
\]
if $D\sim_N D'$ holds for some $N$. 

For a contraction $\sigma\colon X\to Y$ of compact complex varieties and for a $\mathbb{Q}$-divisor $D$ on $X$, we write 
\[
D\sim_{N} 0/Y
\] if $D\sim_N \sigma^* D'$ holds for some $\mathbb{Q}$-divisor $D'$ on $Y$. We write 
\[
D\sim_{\mathbb{Q}} 0/Y
\]
if $D\sim_{N} 0/Y$ holds for some $N$. We say that a divisor $D$ on $X$ is $\sigma$-vertical, 
if~\mbox{$\sigma(D)\neq Y$} holds.  

We say that a divisor $D=\sum a_i D_i$ with $a_i\in\mathbb{Z}$ has multiplicity $m$ if $m$ is the greatest common divisor of $a_i$.%, so we can write $D=mD_1$. 

Given a divisor $D$ on a normal projective variety $Y$, by $[D]$ we will denote its linear equivalence class.  
By $\mathrm{Aut}(Y, [D])$ we denote the subgroup of all automorphisms of~$Y$ which preserve the divisor class $[D]$. Furthermore, if $\Delta$ is a $\QQ$-divisor on $Y$, we denote by $\mathrm{Aut}(Y, \Delta)$ the subgroup of all automorphisms of $Y$ which preserve $\Delta$ as a $\mathbb{Q}$-divisor. This means that $\mathrm{Aut}(Y, \Delta)$ preserves the support of $\Delta$ as a set of points, but the components of $\Delta$ with the same coefficients may be permuted by this group.

\subsection{Pairs and singularities}   By a \emph{contraction} we mean a proper
morphism~\mbox{$f\colon X \to Y$} of normal complex varieties such that~$f_*\oo_X =
\oo_Y$. In particular, $f$ is surjective and has connected fibers. A
\emph{fibration} is defined as a contraction~\mbox{$f\colon X\to Y$} such
that $\dim Y<\dim X$.

A \emph{pair} (resp., a \emph{sub-pair}) $(X, B)$ consists
of a normal complex variety $X$ and a  $\mathbb{Q}$-divisor $B$, called a \emph{boundary} (resp., a \emph{sub-boundary}), with
coefficients in $[0, 1]$ (resp., in $(-\infty, 1]$) such that $K_X + B$ is $\mathbb{Q}$-Cartier.  
Let $\phi\colon W \to X$ be a
log resolution of~\mbox{$(X,B)$} and let 
\[
K_W +B_W = \phi^*(K_X +B)
\]
be the log pull-back of $(X, B)$. 
The \emph{log discrepancy} of a prime divisor $D$ on $W$ with respect
to $(X, B)$ is $1 - \mathrm{coeff}_D B_W$ and it is denoted by $a(D, X, B)$. We
say that~\mbox{$(X, B)$} is lc if $a(D, X, B)\geq 0$  for every $D$. 

Assume that $(X, B)$ is an lc sub-pair. 
%Given an lc sub-pair $(X,D)$, 
We denote by $\lct(X,B;D)$ the log canonical threshold of $(X,B)$ with respect to an effective $\mathbb{Q}$-divisor $D$:
\[
\mathrm{lct} (X, B; D) = \mathrm{sup} \{ \lambda \in \mathbb{Q}\ |\ (X, B + \lambda D)\ \text{is lc} \}. 
\]

\subsection{Pluricanonical representation}
\label{subsec-pluricanonical-rep} 
Let $X$ be a compact complex manifold. Assume that for the Kodaira dimension of $X$ we have $\kappa(X)\ge 0$. For any $m\ge 1$, consider the map
\[
\sigma_m\colon X \dasharrow \mathbb{P}(\mathrm{H}^0(X, \oo_X(mK_X))^\vee).
\]
For a sufficiently big and divisible $m$, we have $\dim \sigma_m(X)=\kappa(X)$. For such $m$, we put $\sigma=\sigma_m$. Thus, $\sigma$ is a well defined meromorphic map (up to the natural equivalence of meromorphic maps). 

\begin{remark}
\label{rem-ueno-connected}
By \cite[Lemma 5.6, Lemma 5.8]{Ue75} for a sufficiently big and divisible $m$ we have that the closures $Y_m$ of the images $\sigma_m(X)$ are bimeromorphic to each other, and the closure of a typical fiber of $\sigma_m$ is connected (and, in fact, irreducible). However, it is not clear whether $Y_m$ is normal. In the projective case, this can be proved using the finite generatedness of the canonical ring, see~\mbox{\cite[Corollary 1.1.2]{BCHM}}. On the other hand, there exists an example of a compact complex manifold (as well as an example of a projective variety with non-lc singularities) whose canonical ring is not finitely generated, see \cite[\S4]{Wilson}. 
\end{remark}
We have the following 

\begin{proposition}[{see e.g. \cite[Lemma 6.3]{Ue75}}]
\label{prop-action-bimer}
The group $\mathrm{Bim}(X)$ acts on $\sigma(X)$ biholomorphically.
\end{proposition}

Put $\mathcal{V}=\mathrm{H}^0(X, \oo_X(mK_X))$ and consider the induced homomorphism
\begin{equation*}
%\label{pluricanonical-rep}
\rho = \rho_m\colon \mathrm{Bim}(X)\to \mathrm{GL}(\mathcal{V}),
\end{equation*}
which is called a \emph{pluricanonical representation} of $\mathrm{Bim}(X)$.   
Consider the natural exact sequence of groups
\[
1 \to \mathbb{C}^\times \to \mathrm{GL}(\mathcal{V}) \xrightarrow{p} \mathrm{PGL}(\mathcal{V})\to 1.
\]
Then by the \emph{projective pluricanonical representation} of $\mathrm{Bim}(X)$ we mean the homomorphism 
\[
\overline{\rho}=p\circ \rho\colon \mathrm{Bim}(X)\to \mathrm{PGL}(\mathcal{V}).
\] 
Observe that the action of $\overline{\rho}(\mathrm{Bim}(X))$ on $\sigma(X)$ is faithful. We will denote $\overline{\rho}(\mathrm{Bim}(X))$ by $\Gamma$ and identify it with a subgroup of the automorphism group of $Y=Y_m$.  

The following example shows that we may not have finite $\rho(\mathrm{Bim}(X))$ when $X$ is a compact complex manifold.% $0\le \kappa(X)\le\dim X - 1$.

\begin{example}
\label{example-infinite}
By \cite{MS08}, there exists a non-algebraic K3 surface $X$ 
that admits a non-symplectic automorphism of infinite order acting on a holomorphic two-form via multiplication by a complex number which is not a root of unity. Consequently, $\rho(\mathrm{Bim}(X))$ is not a finite group. Note also that by \cite[Remark 14.6]{Ue75}, there exists a $3$-dimensional complex torus such that the group $\rho(\mathrm{Bim}(X))$ is infinite.
\end{example}

We formulate the following result for later use. 
\begin{lemma}[{\cite[Lemma 2.9]{PSh21b}}]
\label{lem-fiber-non-negative-kodaira-dim}
Let $X$ and $Y$ be compact complex varieties and let 
$$
\sigma\colon X \to Y
$$ 
be a dominant holomorphic map. Suppose that $\kappa(X)\ge 0$. Let $F$ be a typical fiber of $\sigma$, and let $F'$ be an
irreducible component of $F$. Then $\kappa(F')\ge 0$.
\end{lemma}

\section{Preparations}
In this section, we collect some auxiliary results that will be used in the proof of Theorem \ref{main-theorem}. 

\begin{lemma}[{\cite[Lemma 2.3]{Br22}}]
\label{lem-brion}
Let $D$ be a big divisor on a normal projective variety $X$. Then $\mathrm{Aut}(Y, [D])$ is a linear algebraic group.
\end{lemma}

\begin{remark}
Let $\mathrm{Aut}(Y, [D]_{\mathbb{Q}})$ be the group of all automorphisms of $Y$ which preserve the class of $\mathbb{Q}$-linear equivalence of a $\mathbb{Q}$-divisor $D$. Then $\mathrm{Aut}(Y, [D]_{\mathbb{Q}})$ may not be an algebraic group even if $D$ is a big integral divisor. For instance, if $P$ is a point on an elliptic curve $Y$, then, up to a finite index, the group $\mathrm{Aut}(Y, [P]_{\mathbb{Q}})$ coincides with the group of torsion points in $\mathrm{Pic}^0(Y)$. 
\end{remark}

Using Lemma \ref{lem-brion}, we deduce

\begin{lemma}
\label{lem-linear-alg-group}
Let $Y$ be a normal projective variety, and $\Delta$ be a $\QQ$-divisor 
such that~\mbox{$K_Y+\Delta$} is a big $\mathbb{Q}$-Cartier $\mathbb{Q}$-divisor. Then the group $\mathrm{Aut}(Y, \Delta)$ is a linear algebraic group. 
\end{lemma}

\begin{proof}
Let $N$ be a positive integer such that $N\Delta$ is integral.  
By Lemma \ref{lem-brion} the group $\mathrm{Aut}(Y,[N(K_Y+\Delta)])$ is a linear algebraic group. 
Observe that the classes $K_Y$ and $NK_Y$ are preserved by all automorphisms of $Y$. Hence we have 
\[
\mathrm{Aut}(Y,[N(K_Y+\Delta)])=\mathrm{Aut}(Y,[N\Delta]).
\]
Since 
$$
\mathrm{Aut}(Y, N\Delta)=\mathrm{Aut}(Y, \Delta)
$$ 
is a closed subset of a linear algebraic group $\mathrm{Aut}(Y,[N\Delta])$, it is a linear algebraic group itself.
\end{proof}

The next proposition is just 
\cite[Lemma~2.1]{FeiHu} adapted for our purposes;
see also~\mbox{\cite[Theorem~1.1]{FeiHu}} for further results of the same kind. 
We refer the reader to~\mbox{\cite[Theorem 1.2]{FG14}} for a stronger statement in the case of lc pairs. 

\begin{proposition}
\label{prop-group-preserves-pair}
Let $Y$ be an irreducible projective variety, and $\Delta$ be a $\mathbb{Q}$-divisor such that $(Y, \Delta)$ is an lc sub-pair. Assume that $K_Y+\Delta$ is a big $\QQ$-divisor. 
Then the group $G=\mathrm{Aut}(Y, \Delta)$ is finite. 
%Let $G\subset \Aut(Y)$ be a group which preserves the divisor $\Delta$. 
%Then $G$ is finite.
\end{proposition}

\begin{proof}
%Since $G\subset \mathrm{Aut}(Y,\Delta)$, we may assume that $G=\mathrm{Aut}(Y, \Delta)$. 
By Lemma \ref{lem-linear-alg-group}, we see that $G$ is a linear algebraic group. 
Let $(\widetilde{Y}, \widetilde{\Delta})$ be the log pull-back of $(Y, \Delta)$ via a $G$-equivariant log resolution $f\colon \widetilde{Y}\to (Y,\Delta)$, see 
for instance~\mbox{\cite[Theorem 13.2]{BM}}. 
Put 
$$
\widetilde{\Delta}=B-E,
$$ 
where $B$ and $E$ are effective $\mathbb{Q}$-divisors. 
Thus we have
\[
K_{\widetilde{Y}} + B = f^*(K_Y+\Delta)+E.
\]
Since $K_Y+\Delta$ is big, we see that $f^*(K_Y+\Delta)$ is big, so $K_{\widetilde{Y}} + B$ is big as the sum of a big $\mathbb{Q}$-divisor and an effective $\mathbb{Q}$-divisor. Also note that the pair $(\widetilde{Y}, B)$ is lc.
% (cf. Lazarsfeld 2.2.26). 

Passing to a finite index subgroup of $G$, we may assume that the group~$G$ preserves the pair $(\widetilde{Y}, B)$ and all the irreducible components of the divisor $B$.  
Let $G^0$ be the connected component of identity in $G$. 
To prove that $G$ is finite, it is 
enough to check that $G^0$ is trivial. 

Suppose that it is non-trivial.
It is well known that in this case $G^0$ contains a subgroup $F$ isomorphic either to the multiplicative group 
$\CC^\times$, or to the additive group $\CC^+$, see for 
instance~\mbox{\cite[Lemma~7.2]{ChenShramov}}. 
The closures of general orbits of $F$ are rational curves which cover a Zariski open subset of $Y$.
Let $C$ be a general curve in this family. Then 
$$
BC\le \lceil B\rceil C,
$$
because $C$ is not contained in the support of $B$.
Furthermore, $C$ meets~$\lceil B\rceil$
at most two times; that is, the degree of the support of the pull-back of $\lceil B\rceil$
to the normalization of $C$ is at most~$2$. 
Hence, by~\mbox{\cite[Lemma~5.11]{KeelMcKernan}} one has 
\[
(K_{\widetilde{Y}}+B)C\le (K_{\widetilde{Y}}+\lceil B\rceil)C\le 0,
\]
which gives a contradiction, because $K_{\widetilde{Y}}+B$ is big.
Therefore, $G^0$ is trivial, and $G$ is finite. 
\end{proof}

\begin{remark}
The assumption that $(X, \Delta)$ is lc is neccessary for 
Proposition~\ref{prop-group-preserves-pair} to hold. %Indeed, the pair $(\mathbb{P}^1, 3P)$ admits an infinite group of automorphisms. 
Indeed, the pair $(\PP^2, L_1+L_2+L_3+L_4)$, where 
$L_1$, \ldots, $L_4$ are four lines sharing a common point, admits an infinite group of automorphisms. 
\end{remark}

We also need the following
auxiliary result.

\begin{proposition}\label{proposition:restriction-to-a-curve}
Let $Y$ be a smooth projective variety of dimension $m\ge 2$,
and let~$\Theta$ be a non-trivial  
divisor on $Y$.
Then there exists a very ample divisor $L$   
on~$Y$, such that a general member $R$ of the linear system $|L|$ is smooth and
irreducible, and the divisor $\Theta$ 
restricts to a non-trivial divisor on $R$.
\end{proposition}

\begin{proof}
Let $L$ be an arbitrary divisor on $Y$. Then   
\[
h^1(Y, \Theta-L)=h^{m-1}(Y, K_Y-\Theta+L)
\]
by Serre duality. Furthermore, we have 
\begin{equation}\label{eq:h1-vanishing}
h^{m-1}(Y, K_Y-\Theta+L)=0
\end{equation}
provided that $L$ is a sufficiently high multiple 
of an ample divisor by Serre vanishing. Let $L$ be a very ample 
divisor with this property, and let $R$ be a general member of 
the linear system $|L|$. Then $R$ is smooth and irreducible by Bertini theorem. 

Put $\Theta_R=\Theta\vert_R$.
Assume that the divisor $\Theta_R\in\Pic(R)$ is trivial. 
Then 
$$
h^0(R, \Theta_R)=1.
$$
On the other hand, from 
the restriction exact sequence we obtain an exact sequence 
$$
H^0(Y,\Theta)\to H^0(R, \Theta_R)\to H^1(Y, \Theta-L).
$$
By~\eqref{eq:h1-vanishing}, the first homomorphism in this sequence 
is surjective, and hence
$$
h^0(Y,\Theta)>0.
$$
This means that $\Theta$ is linearly equivalent to an effective 
divisor. Since $\Theta$ is non-trivial, we conclude that 
$$
\Theta_R\cdot L^{m-2}=\Theta\cdot L^{m-1}>0,
$$
which implies that $\Theta_R$ is non-tirivial.
The obtained contradiction completes the proof of the proposition.
\end{proof}

\begin{corollary}\label{corollary:restriction-to-a-curve}
Let $Y$ be a smooth projective variety of dimension $m\ge 2$,
and let~$\Theta$ be a non-trivial 
divisor on $Y$.
Then there exists a smooth irreducible curve $C\subset Y$ 
which is an intersection of general very ample divisors such that 
$\Theta$ restricts to a non-trivial divisor on $C$.
\end{corollary}
\begin{proof}
Applying Proposition~\ref{proposition:restriction-to-a-curve}
several times, we obtain very ample divisors $L_i$
and general divisors $R_i\in |L_i|$, $1\le i\le m-1$, such that the intersection 
$$
C=R_1\cap\ldots\cap R_{m-1}
$$
is a smooth irreducible curve, and $\Theta$ restricts to a non-trivial 
divisor on $C$. 
\end{proof}

We will use the following obvious remark without explicit reference. 

\begin{remark}
\label{rem-restriction-of-two-divisors}
Let $Y$ be a smooth projective variety of dimension $m\ge 2$,
and let~$\Theta$ and $\Theta'$ be two $\mathbb{Q}$-divisors on $Y$. Assume that for a general curve
\[
C=R_1\cap\ldots\cap R_{m-1},
\]
where $R_i$ are very ample divisors on $Y$, the $\QQ$-divisors $\Theta|_C$ and $\Theta'|_C$ 
coincide. Then~\mbox{$\Theta=\Theta'$}, i.e. $\Theta$ and $\Theta'$ coincide 
as $\QQ$-divisors on $Y$. 
\end{remark}

\section{Canonical bundle formula}
\label{sec-canonical-bundle}
In this section, we establish a version of the canonical bundle formula in relative dimension $1$  which works for proper morphisms from a complex variety to a normal projective variety, see Proposition \ref{prop:canonical-bundle-formula}. 
Our exposition goes in parallel with the case of projective morphisms between projective varieties, see e.g. \cite{PS09}. 
In fact, elementary properties of the discriminant and moduli $\mathbb{Q}$-divisors defined below are proven in the same way as in the projective case. 
The proof of Proposition \ref{prop:canonical-bundle-formula} relies on the reduction to the case of a relatively minimal elliptic fibration from a compact complex surface to a curve, in which case we apply a version of the famous Kodaira formula, see Theorem~\ref{theorem:Kodaira}. 
We also formulate a version of the conjecture of Prokhorov and Shokurov 
\cite[Conjecture 7.13]{PS09} in the analytic setting, see Conjecture~\ref{conj-analytic-PS}.
We prove it in the case when the relative dimension of the fibration is $1$ in 
Proposition~\ref{prop-proof-of-conjecture}.

\begin{definition}
Let $X$ be a normal compact complex variety with a 
contraction~\mbox{$\sigma\colon X\to Y$} over a smooth
projective variety $Y$. 
(Note that we do not assume $\sigma$ to be a projective morphism!)  
Suppose that there exists a $\mathbb{Q}$-divisor $R$ on $X$ such that 
\begin{equation*}
K_X+R\sim_{\mathbb{Q}} 0/Y. 
\end{equation*}
%where $H$ is a $\QQ$-divisor on $Y$. 
If $(X,R)$ is an lc sub-pair, then we say that $\sigma\colon (X,R)\to Y$ is an \emph{lc-trivial contraction}. If moreover $\sigma$ is a fibration, we say that it is an \emph{lc-trivial fibration}.
\end{definition}

\begin{definition}
\label{defin-discriminant}
Let $\sigma\colon (X,R)\to Y$ be an lc-trivial contraction. We define the $\mathbb{Q}$-divisor $\Delta=\Delta(\sigma,R)$ on $Y$ by the formula
\begin{equation}
\label{eq-defin-discriminant-divisor}
\Delta=\sum\limits_{D\subset Y} (1-\lct(Y,R;\sigma^*D)) D, 
\end{equation}
where $D$ runs through the set of all prime divisors on $Y$, and the log-canonical threshold is computed over the typical point of $D$. Then $\Delta$ is called the \emph{discriminant $\mathbb{Q}$-divisor} of $\sigma$. We write $\Delta=\Delta(\sigma, R)$ 
to emphasize that $\Delta$ depends on $\sigma$ and $R$. In the case when $R=0$ we simply write $\Delta=\Delta(\sigma)$. 
Furthermore, there exists a $\mathbb{Q}$-divisor $M$ on $Y$, called the \emph{moduli $\mathbb{Q}$-divisor} of $\sigma$, such that
\begin{equation}
\label{eq-defin-discriminant-and-moduli-divisor}
K_X + R \sim_{\mathbb{Q}} \sigma^*(K_Y + \Delta + M),
\end{equation}
cf. \cite[Construction 7.5]{PS09}.
\end{definition}

Note that the discriminant $\mathbb{Q}$-divisor $\Delta$ is defined as a $\QQ$-divisor on~$Y$, and it is effective if $R$ is effective. On the other hand, the moduli $\mathbb{Q}$-divisor $M$ is defined only up to $\QQ$-linear equivalence. 

\begin{remark}
\label{rem-semi-additivity}
Suppose that $\sigma\colon (X, R)\to Y$ is an lc-trivial contraction over~$Y$ with discriminant $\mathbb{Q}$-divisor $\Delta$ and moduli $\mathbb{Q}$-divisor $M$, and 
$B$ is a $\QQ$-divisor on~$Y$ such that $(X, R+\sigma^*B)$ is an lc subpair. Then 
$$
\sigma\colon (X, R+\sigma^*B)\to Y
$$ 
is an lc-trivial contraction with discriminant $\mathbb{Q}$-divisor $\Delta + B$ and moduli $\mathbb{Q}$-divisor $M$. This property of the discriminant $\mathbb{Q}$-divisor is called the \emph{semi-additivity property}, cf.~\mbox{\cite[Lemma 7.4(ii)]{PS09}}.
\end{remark}

\begin{remark}
\label{remark:basechange-property}
Let $\sigma\colon (X, R)\to Y$ be an lc-trivial contraction over $Y$, and~
\mbox{$\psi\colon Y'\to Y$} be a generically finite morphism from a normal projective variety $Y'$. Let $X'$ be the normalization of the main component of the fiber product~\mbox{$X\times_Y Y'$}. Then $X'$ fits into the commutative diagram 
\begin{equation*}
\begin{tikzcd}
X' \arrow{d}{\sigma'} \arrow{r}{\phi}  & X \arrow{d}{\sigma} \\
Y' \arrow{r}{\psi} & Y  
\end{tikzcd}
\end{equation*}
We define a $\QQ$-divisor $R'$ on $X'$ by the formula
\[
K_{X'}+R' \sim_{\QQ} \phi^*(K_X + R). 
\]
Then $\sigma'\colon (X',R')\to Y'$ is also an lc-trivial contraction. 
If $\psi$ is birational, then~\mbox{$\phi_*M_{Y'} = M_Y$} and~\mbox{$\phi_*\Delta_{Y'} = \Delta_Y$}, so these collections of data form b-divisors, see for instance \cite[Section 1.2]{Am04} for the proof in the case when $\sigma$ is a morphism of projective varieties. However, the argument of \cite{Am04}
works in our setting as well. This property of the discriminant $\mathbb{Q}$-divisor is called the \emph{base change property}.
\end{remark}

Next, we make an observation on the behavior of the discriminant $\mathbb{Q}$-divisor under restriction 
to a hyperplane section. 

\begin{lemma}[{cf. \cite[Remark 7.3(i)]{PS09}}]
\label{lem-delta-restricted}
Let $X$ be a compact complex manifold  with a 
contraction $\sigma\colon X\to Y$ 
over a smooth projective variety $Y$.  
Let $R$ be a $\sigma$-vertical $\mathbb{Q}$-divisor on $X$. 
Let $\Delta=\Delta(\sigma,R)$ be the discriminant $\mathbb{Q}$-divisor of $\sigma$. 
Let~$H$ be a general hyperplane section of $Y$, and let $V=\sigma^{-1}(H)$. Put $R_V=R|_V$. 
Let~\mbox{$\Delta_H=\Delta(\sigma_V,R_V)$} be the discriminant $\mathbb{Q}$-divisor of the contraction 
$$
\sigma_V=\sigma|_V\colon V\to H.
$$
Then $\Delta|_H=\Delta_H$. 
\end{lemma}

\begin{proof}
Let $\lambda\in\mathbb{Q}_{\geq 0}$. Since $H$ is general, we see that the pair $(X, R + \lambda \sigma^*Z)$ is lc over a typical point of a prime divsor $Z\subset Y$ 
if and only if the pair $(X,  R + V + \lambda \sigma^*Z)$ is lc over a typical point of $Z|_H$. By inversion of adjunction, this is equivalent to the condition that 
the pair~\mbox{$(V,R_V + \lambda \sigma_V^* Z|_H)$} is lc. Thus,
\[
\mathrm{lct}(X, R; \sigma^*Z) = \mathrm{lct}(V, R_V;\sigma_V^*Z|_H).
\]
This implies that $\Delta|_H=\Delta_H$. 
\end{proof}

In this section, we will be mostly interested in elliptic fibrations. Let us start with a few observations on elliptic fibrations in dimension $2$. 

Let $S$ be a smooth compact complex surface with a relatively minimal
fibration~\mbox{$\sigma\colon S\to C$} over a curve $C$ whose typical fiber is an elliptic curve. 
According to Kodaira's classification (cf. \cite[\S\,V.11]{BHPVdV04}), a fiber $F$ of $\sigma$  can have one of the following types:
\begin{center}
\begin{longtable}{|p{0.14\textwidth}|p{0.13\textwidth}|p{0.09\textwidth}|p{0.09\textwidth}|p{0.10\textwidth}|p{0.22\textwidth}|}
\hline
Type & $I_n, n\geq 0$ & $II$ & $III$ & $IV$ & $\prescript{}{m}I_n$, $n\geq 0$, $m\geq 1$\\ \hline
$1-\mathrm{lct}(S,F)$ & $0$ & $1/6$ & $1/4$ & $1/3$ & $1-1/m$ \\ \hline
Monodromy & 
${\tiny\begin{pmatrix}
    1 & n \\
    0 & 1
\end{pmatrix}}$
& 
${\tiny\begin{pmatrix}
    1 & 1 \\
    -1 & 0
\end{pmatrix}}$
& 
${\tiny\begin{pmatrix}
    0 & 1 \\
    -1 & 0
\end{pmatrix}}$
& 
${\tiny\begin{pmatrix}
    0 & 1 \\
    -1 & -1
\end{pmatrix}}$
& 
\vphantom{
${\tiny\begin{pmatrix}
    -1 & -1 \\
    1 & 0
\end{pmatrix}}_j^{A}$
}
 \\
\hline
\hline 
Type & $I_n^*, n\geq 0$ & $II^*$ & $III^*$ & $IV^*$ & \\ \hline
$1-\mathrm{lct}(S,F)$ & $1/2$ & $5/6$ & $3/4$ & $2/3$ & \\ \hline
Monodromy & 
${\tiny\begin{pmatrix}
    -1 & -n \\
    0 & -1
\end{pmatrix}}$
& 
${\tiny\begin{pmatrix}
    0 & -1 \\
    1 & 1
\end{pmatrix}}$
& 
${\tiny\begin{pmatrix}
    0 & -1 \\
    1 & 0
\end{pmatrix}}$
& 
${\tiny\begin{pmatrix}
    -1 & -1 \\
    1 & 0
\end{pmatrix}}$
&
\vphantom{
${\tiny\begin{pmatrix}
    -1 & -1 \\
    1 & 0
\end{pmatrix}}_j^{A}$
}
\\ 
\hline
\caption{Fibers of a relatively minimal elliptic fibration}
\label{table:monodromy} 
\end{longtable}
\end{center}
Here, in the second row we indicate the coefficient of $\sigma(F)$ in the discriminant $\mathbb{Q}$-divisor $\Delta(\sigma)$ on $C$; in the third row we indicate the monodromy matrix which acts on the middle cohomology of a typical fiber. 
The multiple fibers have type 
$\prescript{}{m}I_n$ for~\mbox{$n\geq 0$} and $m\geq 1$. 

The following proposition is well known to experts, cf. \cite[2.6]{Fujita}. 

\begin{proposition}
\label{prop-monodromy-determines-everything}
Let $S$ be a smooth compact complex surface with a relatively minimal
fibration $\sigma\colon S\to C$ over a curve $C$ whose typical fiber is an elliptic curve.
For a non-multiple fiber $F$ of $\sigma$, the value $1-\mathrm{lct}(S,F)$ is uniquely determined by the conjugacy class of the monodromy matrix in $\mathrm{SL}_2(\mathbb{Z})$.
\end{proposition}
\begin{proof}
The monodromy matrix for type $I_n$ is unipotent, and for type $I_n^*$ is quasi-unipotent but not unipotent; it has infinite order for $n\geq 1$.  
For the types $I_0^*$, $II$, $II^*$, $III$, $III^*$, $IV$ and $IV^*$, the matrices have finite order.

The monodromy matrices for type $I_0^*$ has trace $-2$, for types $II$ and $II^*$ have trace $1$, for types $III$ and $III^*$ have trace $0$, for types $IV$ and $IV^*$ have trace $-1$. 

Finally, it is straightforward to check that the monodromy matrices for types~$II$ and~$II^*$ (resp., $III$ and $III^*$; $IV$ and $IV^*$) are conjugate in $\mathrm{GL}(\mathbb{Z})$, but not in~$\mathrm{SL}(\mathbb{Z})$.  This shows that the type of a non-multiple fiber, and hence the value of~\mbox{$1-\mathrm{lct}(S,F)$}, is uniquely determined by the monodromy.
\end{proof}

Recall the following classical result.

\begin{theorem}[{see \cite[Theorem~8.2.1]{Kollar}, \cite[\S2]{Fujita}, \cite[Theorem 6.1]{Ue73},
or~\mbox{\cite[Theorem~12]{Kod64}}}]
\label{theorem:Kodaira}
Let $S$ be a smooth compact complex surface with a relatively minimal
fibration~\mbox{$\sigma\colon S\to C$} over a curve $C$ whose typical fiber is an elliptic curve. 
Let $\mu$ be the least common multiple of the multiplicities of all multiple fibers of $\sigma$. 
Let $j\colon C\to\PP^1$ be the map defined by the $j$-invariant 
of the fibers of $\sigma$. 
Then
\begin{equation}
\label{eq-fujita-ueno}
K_S\sim_{12\mu} \sigma^*\left(K_C+\Delta+M\right), 
\end{equation}
where $M\sim_{12}\frac{1}{12}j^*Q$ for a point $Q\in \PP^1$, and $\Delta=\Delta(\sigma)$ is the discriminant $\mathbb{Q}$-divisor.  
\end{theorem}

\begin{proof}[Sketch of proof]
We explain the main idea of the proof given in \cite[\S{2}]{Fujita}, where it is assumed that the surface $S$ is projective, but this assumption is not actually used. 

We start with case when $\sigma$ has no multiple fibers. In this case one can write
\begin{equation}
\label{eq-is-line-bundle}
\omega_S \sim \sigma^*(\omega_C \otimes \sigma_* \omega_{S/C}),
\end{equation}
where $\omega_S$ and $\omega_C$ are the canonical line bundles on $S$ and $C$, respectively, and 
$$
\omega_{S/C}=\omega_S\otimes \sigma^* \omega_C^{-1}
$$ 
is the relative canonical line bundle. We want to show that $12\Delta$ and $12M$ are integral divisors, and 
\begin{equation}
\label{eq-isom-fujita}
(\sigma_* \omega_{S/C})^{\otimes 12}\cong \mathcal{O}_C(12(\Delta+M)).
\end{equation}
To this aim, we 
construct a holomorphic section of the line bundle $(\sigma_*\omega_{S/C})^{\otimes12}$ and understand its zeroes on $C$. Let $U\subset C$ be the open set over which $\sigma$ is a smooth morphism. Put $V=\sigma^{-1}(U)$ and 
$$
\varsigma = \sigma|_V\colon V\to U.
$$
Let $\overline{\varsigma}\colon \mathrm{Jac}(\varsigma)\to U$ be the Jacobian fibration of $\varsigma$, see e.g. \cite[\S\,V.9]{BHPVdV04}. Then one has 
\[
\varsigma_* \omega_{V/U} \cong \overline{\varsigma}_*\omega_{Jac(\varsigma)/U}.
\] 
Since $\overline{\varsigma}$ admits a section, its fibers can be brought to the Weierstrass form. In particular, one has a well-defined cusp form $\delta(z)$ of weight six on $U$, which provides a holomorphic section $s$ of the line bundle $(\overline{\varsigma}_*\omega_{Jac(\varsigma)/U})^{\otimes12}$, and thus a meromorphic section of $(\sigma_*\omega_{S/U})^{\otimes12}$. Now, we use the extension method of \cite[\S6]{Ue73} to obtain the isomorphism~\eqref{eq-isom-fujita} in the case when there are no multiple fibers. 

Let $z_0\in C\setminus U$. 
Assume that the fiber $F=\sigma^{*}z_0$ is not of type $I_n$ or $I_n^*$ for~\mbox{$n\geq 1$}. 
Then the extension of $\delta(z)$ to $z_0$ is holomorphic and has a zero 
of order~\mbox{$12(1-\mathrm{lct}(S, \sigma^* z_0))$} at~$z_0$. 
If $F$ is of type $I_n$ for $n\geq 1$, then the extension of~$\delta(z)$ to $z_0$ is holomorphic and has a zero of order $n$, while the meromorphic function $j$ has a pole of order $12n$ at~$z_0$, and the coefficient 
at $z_0$ in $\Delta$ equals~$0$.  
If $F$ is of type~$I_n^*$ for $n\geq 1$, then the extension of~$\delta(z)$ to $z_0$ is holomorphic and has a zero of order~\mbox{$6+n$}, while $j$ has a pole of order $12n$ at $z_0$, 
and the coefficient at $z_0$ in $12\Delta$ equals~$6$. Furthermore, $j$ has no other poles. 
Hence, the extension of $\delta(z)$ provides a section of the line bundle 
$$
(\sigma_* \omega_{S/C})^{\otimes 12}\otimes \mathcal{O}_C(-12(\Delta+M))
$$
without zeros and poles, which gives the desired isomorphism \eqref{eq-isom-fujita}.

Now consider the case when $\sigma$ has multiple fibers. 
Assume that $\sigma^{*}z_0$ for $z_0\in C$ has type $\prescript{}{m}I_n$ for $n\geq 0$, $m\geq 1$. 
We are going to establish an isomorphism 
\begin{equation}
\label{eq-isom-fujita-2}
(\sigma_* \omega_{S/C})^{\otimes 12\mu}\cong \mathcal{O}_C(12\mu(\Delta+M)).
\end{equation}
To this aim, 
we apply (the inverse of) a logarithmic transformation which replaces a fiber $F$ of type $\prescript{}{m}I_n$ by a fiber $F'$ of type $I_n$, see~\mbox{\cite[\S\,V.13]{BHPVdV04}}. 
After eliminating all multiple fibers, we can apply the canonical bundle formula obtained above, 
and then compute the additional input of multiple fibers. Overall, from local computation it follows that the extension of $\delta(z)^m$ to $z_0$ is holomorphic and has a zero of order~\mbox{$12(m-1)+12mn$} at~$z_0$. 
Hence the extension of $\delta(z)^{\mu}$ has a zero of order 
$$
\frac{\mu}{m}\cdot \big(12(m-1)+12mn\big)=12\mu\big((1-\frac{1}{m})+n) 
$$
at $z_0$. On the other hand, the coefficient at $z_0$ in $\Delta$
equals $1-1/m$, while $j$ has a pole of order $n$ at $z_0$. 
Therefore, $\delta(z)^{\mu}$ extends to a section of the line bundle
$$
(\sigma_* \omega_{S/C})^{\otimes 12\mu}\otimes \mathcal{O}_C(-12\mu(\Delta+M))
$$
without zeroes and poles. 
This gives us the isomorphism~\eqref{eq-isom-fujita-2}. Now it follows 
from~\eqref{eq-is-line-bundle} that formula~\eqref{eq-fujita-ueno} holds. 
\end{proof}

\begin{remark}
We point out that in the course of the proof of Theorem \ref{theorem:Kodaira}, even if we start with a projective surface, we may have to replace it with a non-projective surface when doing logarithmic transformations to deal with multiple fibers. On the other hand, passing to the Jacobian fibration brings us to the quasi-projective category anyway. 
\end{remark}

%\begin{remark}
%In fact, one can prove that  
%\begin{equation*}
%K_S\sim_{12} \sigma^*\left(K_C+\Delta+M\right);
%\end{equation*}
%the reader can find a more detailed proof in \cite[\S2]{Fujita}. 
%However, formula~\eqref{eq-fujita-ueno} is enough for our purposes. 
%\end{remark}

The following examples illustrate Theorem \ref{theorem:Kodaira} in the case when there are no degenerate fibers of $\sigma$. 

\begin{example}\label{example:Hopf}
Let $\alpha$ and $\beta$ be complex numbers such that $0<|\alpha|, |\beta|<1$, 
and~\mbox{$\alpha^k=\beta^l$} for some positive integers $k$ and $l$. Then the quotient 
of $\CC^2\setminus\{0\}$ by the group $\ZZ$ generated by the transformation 
$$
(z_1,z_2)\mapsto (\alpha z_1,\beta z_2)
$$
is a primary Hopf surface with an elliptic fibration 
$\sigma\colon S\to \PP^1$; we refer the reader to~\mbox{\cite[\S\,V.18]{BHPVdV04}}
for more details on Hopf surfaces. 
Observe that all fibers of~$\sigma$ are isomorphic to each other; in particular, 
$\sigma$ has no degenerate fibers. Thus, in the notation of Theorem~\ref{theorem:Kodaira} one has $\Delta=j^*Q=0$, so that 
$$
K_S\sim_{\QQ}\sigma^* K_{\PP^1}.
$$
\end{example}

\begin{example}\label{example:Kodaira}
Let $S$ be a primary Kodaira surface, see e.g.~\mbox{\cite[\S\,V.5]{BHPVdV04}}. 
Then the algebraic reduction of $S$ 
is an elliptic fibration $\phi\colon S\to E$ without degenerate fibers over an elliptic curve $E$.
Thus, by Theorem~\ref{theorem:Kodaira} we have 
$$
K_S\sim_{\QQ}\sigma^* K_E\sim 0.
$$
\end{example}

\begin{remark}
In fact, in Examples~\ref{example:Hopf} and~\ref{example:Kodaira}
one can show that the canonical divisor of $S$ is linearly equivalent (not just $\QQ$-linearly equivalent) to $\sigma^*K_{\PP^1}$ and $0$, respectively.  
\end{remark}

We proceed with a couple of observations concerning Theorem~\ref{theorem:Kodaira}. 

\begin{remark}
In the notation of Theorem \ref{theorem:Kodaira}, the condition that $\sigma$ is relatively minimal is equivalent to the condition that $K_S$ is $\QQ$-linearly trivial over $C$. 
\end{remark}

The next result generalizes Theorem \ref{theorem:Kodaira} to the case of not necessarily relatively minimal fibrations on complex surfaces, cf. \cite[Example 7.16]{PS09} in the projective setting. 

\begin{proposition}[{cf. \cite[2.10]{Fujita}}]
\label{prop-elliptic-surface-with-divisor}
Let $S$ be a smooth compact complex surface with a 
(not necessarily relatively minimal) fibration $\sigma\colon S\to C$
over a curve $C$ whose typical fiber is an elliptic curve. 
Let $R$ be a $\sigma$-vertical integral divisor on $S$ such that~\mbox{$(S, R)$} is an lc sub-pair. 
Assume that 
\begin{equation}
\label{eq-kodaira-formula-with-divisor-assumption}
K_S+R\sim_{\QQ}0/C.
\end{equation} 
Let $j\colon C\to\PP^1$ be the map defined by the $j$-invariant 
of the fibers of $\sigma$. 
Let $\mu$ be the least common multiple of the multiplicities of all multiple fibers of $\sigma$. 
Then there exists a $\mathbb{Q}$-divisor $\Delta$ on $C$ such that 
\begin{equation}
\label{eq-kodaira-formula-with-divisor}
K_S+R\sim_{12\mu} \sigma^*\left(K_C+M+\Delta\right),
\end{equation}
where $M\sim_{12}\frac{1}{12}j^*Q$ for a point $Q\in \PP^1$, and $\Delta=\Delta(\sigma, R)$ is the discriminant $\mathbb{Q}$-divisor. 
Moreover, $(C, \Delta)$ is an lc sub-pair. 
\end{proposition}

\begin{proof}
First, assume that $\sigma\colon S\to C$ is a relatively minimal elliptic fibration. Then Theorem~\ref{theorem:Kodaira} yields
\[
K_S\sim_{12\mu} \sigma^*(K_C+\frac{1}{12}j^*Q+\Delta_0), 
\]
where $\Delta_0$ is an effective $\QQ$-divisor such that 
$$
\Delta_0=\sum\limits_{P\in C} (1-\lct(S;\sigma^*P)) P.
$$
Then \eqref{eq-kodaira-formula-with-divisor-assumption} implies that $R$ is $\QQ$-linearly trivial over $C$, so that 
$$
R=\sum_{P\in C} c_P \sigma^*(P)
$$ 
for some $c_P\in \mathbb{Q}$. It follows that 
\[
K_S + R \sim_{12\mu} \sigma^*(K_C + \frac{1}{12}j^*Q + \Delta)
\]
where \[
\Delta = \sum_{P\in C}(1-\mathrm{lct}(S;\sigma^*P)+c_P)P = \sum_{P\in C}(1-\mathrm{lct}(S, R;\sigma^*P))P
\]
by Remark \ref{rem-semi-additivity}.  
Hence the formula \eqref{eq-kodaira-formula-with-divisor} holds for a relatively minimal elliptic fibration. 

Now we treat the general case. 
Since in dimension~$2$ any fibration has a relatively minimal model, it is enough to suppose that the formula \eqref{eq-kodaira-formula-with-divisor} holds for a fibration~\mbox{$\sigma\colon S\to C$}, and prove it for a fibration $\sigma_1\colon S_1\to C$ such that $S_1$ is obtained from~$S$ by a composition $\phi\colon S_1\to S$ of blow ups, 
assuming that $\sigma_1$ satisfies~\eqref{eq-kodaira-formula-with-divisor-assumption}, 
i.e.~\mbox{$K_{S_1}+R_1\sim_{\QQ}\sigma_1^*D$} for some $\QQ$-divisor $D$ on $C$. 
So we have the following diagram:
\begin{equation*}
\begin{tikzcd}
S_1 \arrow{d}{\sigma_1} \arrow{r}{\phi}  & S \arrow{d}{\sigma} \\
C \arrow[equal]{r} & C  
\end{tikzcd}
\end{equation*}
Hence $K_{S}+R\sim_{\QQ}\sigma^*D$, where $R=\phi_*R_1$. 
This gives $\phi^*(K_{S_1}+R_1)\sim_{\mathbb{Q}} K_S+R$.  
By assumption, we have
\[
K_{S} + R \sim_{12\mu} \sigma^*(K_C + \frac{1}{12}j^*Q + \Delta), 
\]
where $\Delta$ is defined by the formula
\[
\Delta = \sum_{P\in C}(1-\mathrm{lct}(S, R;\sigma^*P))P.
\]
It follows that 
\[
K_{S_1}+R_1\sim\phi^*(K_{S} + R)\sim_{12\mu}\sigma_1^*(K_C + \frac{1}{12}j^*Q + \Delta)
\]
and 
\[
\mathrm{lct}(S_1, R_1;\sigma_1^*P)=\mathrm{lct}(S, R;\sigma^*P),
\]
thus $\Delta_1=\Delta$. 

The fact that $(C,\Delta)$ is an lc sub-pair is trivial, because the coefficients 
of $\Delta$ do not exceed $1$ by construction. 
\end{proof}

Now we consider varieties of arbitrary dimension.

\begin{proposition}[{cf. \cite[2.15]{Fujita}}]
\label{prop:canonical-bundle-formula}
Let $X$ be a compact complex manifold  with a 
fibration $\sigma\colon X\to Y$ 
over a smooth projective variety $Y$ 
whose typical fiber is an elliptic curve. Let $R$ be a $\sigma$-vertical integral divisor on $X$ such that $(X, R)$ is an lc sub-pair. Assume further that  
\[
K_X+R\sim_{N}0/Y
\]
for some positive integer $N$.   
Then
\begin{equation*}
K_X+R\sim_{N'} \sigma^*\left(K_Y+M + \Delta\right)
\end{equation*}
for some positive integer $N'$, where 
\begin{enumerate}
\item
$\Delta=\Delta(\sigma,R)$ is the discriminant $\mathbb{Q}$-divisor, 
\item
the moduli $\mathbb{Q}$-divisor $M$ is b-nef, 
\item
$(Y, \Delta)$ is an lc sub-pair. 
\end{enumerate}
Moreover, if the $j$-invariant map $j\colon Y\dashrightarrow \mathbb{P}^1$ is a morphism, then $M\sim_{12}\frac{1}{12}j^*Q$ for a point~\mbox{$Q\in\PP^1$}; in particular, $M$ is nef. 
\end{proposition}

\begin{proof}
We shall construct some modifications of $X$ and $Y$ and use the base change property of the discriminant part and the moduli part as in Remark~\ref{remark:basechange-property}. First of all, since the typical fiber of $\sigma$ is an elliptic curve, there is a $j$-invariant map $j\colon Y\dashrightarrow \mathbb{P}^1$. Resolving the indeterminacy of the map $j$, we obtain a morphism $\psi\colon Y_1\to Y$ such that $j\colon Y_1\to \mathbb{P}^1$ is a morphism. We may assume that $Y_1$ is smooth. Consider the following diagram
\begin{equation}
\label{diagram_resolution}
\begin{tikzcd}
X_1 \arrow{d}{\sigma_1} \arrow{r}{\phi}  & X \arrow{d}{\sigma} \\
Y_1 \arrow{r}{\psi} \arrow{d}{j} & Y \arrow[dashed]{d}{j}\\
\mathbb{P}^1 \arrow[equal]{r} & \mathbb{P}^1
\end{tikzcd}
\end{equation}
Here $X_1$ is the normalization of the main component of the fiber product $X\times_Y Y_1$. Resolving singularities, we may assume that $X_1$ is smooth.

We define a $\mathbb{Q}$-divisor $R_1$ on $X_1$ by the formula 
\[
K_{X_1}+R_1\sim \phi^*(K_{X}+R).
\] 
Note that $R_1$ is an integral divisor since $X$ is smooth and $R$ is an integral divisor. 
Then
\[
K_{X_1}+R_1\sim_{N} \sigma_1^*D_1,
\] 
where $D_1=\psi^*D$ is a $\QQ$-divisor on $Y_1$ such that $ND_1$ is an integral divisor. 
We can write
$$
K_{X_1}+R_1\sim_{N} \sigma_1^*(K_{Y_1}+\widetilde{\Delta})
$$
for a $\QQ$-divisor $\widetilde{\Delta}=D_1-K_{Y_1}$
on $Y$. 
Put $M_1=\frac{1}{12}j^*Q$ for a point $Q\in\PP^1$, 
and define~$\Delta_1$ by the formula
\[
\Delta_1=\sum\limits_{Z\subset Y} (1-\lct(X_1,R_1;\sigma^*Z)) Z,
\]
where $Z$ runs through the set of all prime divisors on $Y_1$, and the log-canonical threshold is computed over the typical point of $Z$. 
It follows that $N_1\Delta_1$ is an integral divisor for some $N_1>0$ which depends on the map $\sigma_1$. 

Denote by $\mu$ the least common multiple of the multiplicities of the fibers of~$\sigma_1$, that is, the least common multiple of the multiplicities of all 
pull-backs of prime divisors on~$Y_1$ via~$\sigma_1$. 
Set 
$$
N'=12NN_1\mu.
$$  
We are going to show that $\widetilde{\Delta}\sim_{N'} M_1+\Delta_1$.

Suppose that 
$$
\widetilde{\Delta}-\Delta_1 - M_1\not\sim_{N'}0, 
$$
or, in other words,
%\quad \quad \text{so} \quad \quad 
$$
N'(\widetilde{\Delta}-\Delta_1 - M_1)\not\sim 0.
$$ 
According to Corollary~\ref{corollary:restriction-to-a-curve}, there exists
a smooth irreducible curve
$C\subset Y$ which is a complete intersection of general very ample divisors $L_1,\ldots,L_{n-2}$ such that
$$
N'(\widetilde{\Delta}-\Delta_1-M_1)\vert_C\not\sim 0,
$$
or, in other words, 
\begin{equation}\label{eq:restricts-non-trivially}
(\widetilde{\Delta}-\Delta_1-M_1)\vert_C\not\sim_{N'} 0. 
\end{equation} 

Let $S$ denote the preimage of $C$ with respect to~$\sigma$. 
Since $L_i$ are general, we conclude from Bertini theorem, 
see e.g.~\mbox{\cite[Theorem~4.21]{Ue75}}, 
that $S$ is a smooth compact complex surface.
Let $\mathcal{D}$ be the (finite) set of all prime divisors $Z$ on $Y_1$ such that
$$
\lct(X_1,R_1+\sigma^*Z))\neq1.
$$
It also follows from Bertini theorem that 
$C$ is in general position with respect to all the divisors
$D\in\mathcal{D}$. 
Put $R_S = R_1|_S$. 
Hence for the $\QQ$-divisor $\Delta_C=\Delta_1\vert_C$
we have
$$
\Delta_C=\sum\limits_{P\in C} (1-\lct(S,R_S;\sigma^*P)) P,
$$
see Lemma~\ref{lem-delta-restricted}.

By adjunction, one has
$$
K_C\sim (K_{Y_1}+L_1+\ldots+L_{n-2})\vert_C.
$$
Thus, we have  
\begin{multline}\label{eq:elliptic-fibration-long-equivalence} 
K_S + R_S\sim (K_{X_1}+R_1+\sigma^*L_1+\ldots+\sigma^*L_{n-2})\vert_S\\
\sim_{N'}
\big(\sigma_1^*(K_{Y_1}+\widetilde{\Delta}+L_1+\ldots+L_{n-2})\big)\vert_S\\
\sim_{N'}
\sigma_1^*\big((K_{Y_1}+\widetilde{\Delta}+L_1+\ldots+L_{n-2})\vert_C\big)\\
\sim_{N'} \sigma^*\big(K_C+\widetilde{\Delta}\vert_C)\big).
\end{multline}
Note that $R_S$ is an integral divisor. 
In particular, a typical fiber of the fibration~\mbox{$\sigma_1\colon S\to C$} is an elliptic curve. Also, we see that the least common multiple of the multiplicities of the fibers of this fibration equals~$\mu$. 
Thus, by Proposition~\ref{prop-elliptic-surface-with-divisor}
one has 
\begin{equation}\label{eq:elliptic-fibration-restriction-to-curve}
K_S + R_S\sim_{12\mu} \sigma^*(K_C+M_C+\Delta_C), 
\end{equation}
where $M_C\sim_{12}\frac{1}{12}j^*Q$ for a point $Q\in\PP^1$. 
Combining~\eqref{eq:elliptic-fibration-long-equivalence} and~\eqref{eq:elliptic-fibration-restriction-to-curve}, 
we obtain 
$$
\sigma_1^*(\widetilde{\Delta}\vert_C)\sim_{N'} \sigma_1^*(M_C + \Delta_C)=\sigma_1^*\big((M_1+\Delta_1)\vert_C\big).
$$
On the other hand, it follows from projection formula that the map
$$
\sigma_1^*\colon\Pic(C)\to \Pic(S)
$$
is injective. 
This gives 
$$
\widetilde{\Delta}\vert_C\sim_{N'} (M_1+\Delta_1)\vert_C.
$$
The obtained contradiction with~\eqref{eq:restricts-non-trivially} 
shows that $\widetilde{\Delta}\sim_{N'} M_1+\Delta_1$, and  
$$
K_X+R\sim_{N'} \sigma^*\left(K_Y+M + \Delta\right).
$$

The fact that $(C, \Delta)$ is an lc sub-pair follows from \cite[Corollary 7.18]{PS09}. Note that in \cite{PS09} the authors work in the case of projective varieties, however, their argument works in the analytic setting as well.

Note that we have $\psi_*M_1=M$ (and $\psi_*\Delta_1=\Delta$), so $M$ is b-nef. Also, 
from diagram~\eqref{diagram_resolution} it is clear that if the $j$-invariant map $j\colon Y\dashrightarrow \mathbb{P}^1$ is a morphism, then~\mbox{$M\sim_{12}\frac{1}{12}j^*Q$} for a point $Q\in\PP^1$.
\end{proof}

We formulate a version of the conjecture of Prokhorov and Shokurov 
\cite[Conjecture 7.13]{PS09} which is applicable to our setting.

\begin{conjecture}
\label{conj-analytic-PS}
Let $X$ be a compact complex manifold with an 
lc-trivial contraction $\sigma\colon (X,R)\to Y$ over a smooth
projective variety $Y$, where $R$ is a $\mathbb{Q}$-divisor on~$X$. 
(We emphasize that $\sigma$ need not be a projective morphism!) 
Then we have
\[
K_X + R \sim_{\mathbb{Q}} \sigma^*(K_Y+\Delta+M),
\]
where $\Delta=\Delta(\sigma, R)$ is the discriminant $\mathbb{Q}$-divisor, and the moduli $\mathbb{Q}$-divisor $M$ is $b$-semi-ample. 
\end{conjecture}

One could also formulate an effective version of this conjecture, 
cf.~\mbox{\cite[7.13.2, 7.13.3]{PS09}}. 
Also, one can allow mild singularities on $X$ and~$Y$. \begin{proposition}
\label{prop-proof-of-conjecture}
Conjecture \ref{conj-analytic-PS} holds in the case $\dim X-\dim Y\leq 1$. 
\end{proposition}
\begin{proof}
If $\dim X=\dim Y$ then $X$ is projective and $\sigma$ is projective as well. Then the claim follows from \cite[Example 7.8]{PS09}. 

Now consider the case $\dim X=\dim Y+1$.
Assume that $R$ has a $\sigma$-horizontal component. Then $X$ is Moishezon (cf. \cite[Lemma 19, Example 12]{Kollar22}), so there exists a bimeromorphic map $f\colon X_1\to X$ such that $X_1$ is a projective variety. Let~\mbox{$\sigma_1\colon X_1\to Y$} be the induced map. Let   $K_{X_1}+R_1\sim_{\mathbb{Q}}0/Y$ be the crepant pull-back of $K_X+R\sim_{\mathbb{Q}}0/Y$. Observe that $\sigma_1\colon (X_1,R_1)\to Y$ is an lc-trivial fibration. By~\mbox{\cite[Theorem 8.1]{PS09}}, we have 
\[
K_{X_1}+R_1\sim_{\mathbb{Q}}\sigma_1^*(K_Y+\Delta_1+M_1) 
\]
where $\Delta_1=\Delta(X_1,\sigma_1)$ is the discriminant $\mathbb{Q}$-divisor and $M_1$ is the moduli $\mathbb{Q}$-divisor which is b-semi-ample. 
After passing to a higher model of $X_1$, we may assume that~$M_1$ is semi-ample. 
Taking the push-forward, we get  
\[
K_{X}+R\sim_{\mathbb{Q}}\sigma^*(K_Y+\Delta+M), 
\]
where $f_*\Delta_1=\Delta=\Delta(X,\sigma)$ (cf. Remark~\ref{remark:basechange-property}), and $M=f_*M_1$ is b-semi-ample. 

Now assume that $R$ has no $\sigma$-horizontal components, so $R$ is $\sigma$-vertical. Then we conclude by Proposition \ref{prop:canonical-bundle-formula}. 
\end{proof}

In the projective setting, an effective version of Conjecture \ref{conj-analytic-PS} is proven when the relative dimension of $\sigma$ is at most $2$, see \cite[Theorem 1.4]{A++}. 
Furthermore, in the projective setting, for arbitrary relative dimension it is known that the weaker property always holds: namely, that the moduli $\mathbb{Q}$-divisor $M$ is b-nef, 
see~\mbox{\cite[Theorem~8.3.7]{Kollar}}. 
If the base of the fibration is a curve, this implies that the conjecture holds in this case in the projective setting, see \cite[Theorem 0.1]{Am04}. 
For a survey of other results in this direction, we refer the reader to~\cite{FlorisLazic}. 

However, the b-nef property fails in the complex setting even in relative dimension~$2$ as the following example shows. 

\begin{example}[{\cite[\S10]{Atiyah}}]
\label{exam-atiyah}
Let $Y=\mathbb{P}^1$. Pick two non-proportional global sections $s_1, s_2$ of $\oo_Y(1)$. Then $(s_1, s_2)$ is a section of $\oo_Y(1)\oplus \oo_Y(1)$ which is nowhere zero. 
Put
\[
I_1 = {\begin{pmatrix}
    1 & 0 \\
    0 & 1
\end{pmatrix}}, \quad
I_2 = {\begin{pmatrix}
    0 & 1 \\
    -1 & 0
\end{pmatrix}}, \quad
I_3 = {\begin{pmatrix}
    0 & \sqrt{-1} \\
    \sqrt{-1} & 0
\end{pmatrix}}, \quad
I_4 = {\begin{pmatrix}
    \sqrt{-1} & 0 \\
    0 & -\sqrt{-1}
\end{pmatrix}}.
\]
Let $L$ be the total space of $\oo_Y(1)\oplus\oo_Y(1)$, and let 
\[
\Lambda = \left( I_1{\begin{pmatrix}
    s_1 \\
    s_2 \end{pmatrix}},
    I_2{\begin{pmatrix}
    s_1 \\
    s_2 \end{pmatrix}},
    I_3{\begin{pmatrix}
    s_1 \\
    s_2 \end{pmatrix}},
    I_4{\begin{pmatrix}
    s_1 \\
    s_2 \end{pmatrix}}
    \right).
\]
Thus, the fiber of $\Lambda\subset L$ over any point in $\mathbb{P}^1$
is a lattice which has rank $4$. 
Consider the quotient $X = L/\Lambda$, and let $\sigma\colon X\to Y$ be the natural projection. Then $X$ is a compact complex manifold, and every fiber of $\sigma$ is a (smooth) complex torus. 
Hence, $\sigma\colon X\to Y$ is an lc-trivial fibration whose discriminant $\mathbb{Q}$-divisor $\Delta$ vanishes.  
From~\mbox{\cite[Proposition 10]{Atiyah}} it follows that 
\[
K_X \sim \sigma^* (K_{Y}-2Q),
\] 
where $Q\in \mathbb{P}^1$ is a point. Thus the moduli $\mathbb{Q}$-divisor $M\sim -2Q$ is not pseudo-effective, and so it is not b-nef. 
\end{example}

Example \ref{exam-atiyah} shows that Conjecture \ref{conj-analytic-PS} has no chance to be true in the relative dimension at least $2$ for arbitrary compact complex manifolds. However, it would be interesting to figure out if it holds when the total space of the fibration is K\"ahler.

\section{Equivariant fibrations}

In this section, we study bimeromorphic automorphisms of complex manifolds which preserve the structure of a fibration and act on the base of the fibration biregularly.

\begin{definition}
Let $\sigma\colon X\to Y$ be a fibration from a compact complex manifold to a normal projective variety. Put
\[
\mathrm{Bim}(X;\sigma) = \{ \alpha\in \mathrm{Bim}(X)\ |\ \sigma\ \text{is}\ \alpha\text{-equivariant}\,\}.
\] 
Then there is a natural homomorphism 
$
\overline{\sigma}\colon \mathrm{Bim}(X;\sigma)\to \mathrm{Bim}(Y).
$
Put 
\[
\mathrm{Bim}_{\mathrm{reg}}(X;\sigma) = \{ \alpha\in \mathrm{Bim}(X;\sigma)\ |\ \overline{\sigma}(\alpha)\in\mathrm{Aut}(Y)\,\}.
\] 
\end{definition}

By construction, we have a homomorphism 
\[
\overline{\sigma}\colon \mathrm{Bim}_{\mathrm{reg}}(X;\sigma)\to \mathrm{Aut}(Y).
\]
We will work in the following setting.

\begin{assumptions}
\label{setting-equivariant-fibrations}
Assume that
\begin{enumerate}
\item 
$\sigma\colon (X,R)\to Y$ is an lc-trivial fibration from a compact complex manifold to a normal projective variety whose typical fiber is an elliptic curve,  
\item 
$R$ is a $\sigma$-vertical and anti-effective $\mathbb{Q}$-divisor,
\item
$R$ is a simple normal crossing $\mathbb{Q}$-divisor.
\end{enumerate}
\end{assumptions}

Note that the third condition can always be obtained by passing to a higher model of $X$. 
Recall the definition of the discriminant $\mathbb{Q}$-divisors $\Delta(\sigma)=\Delta(\sigma,0)$ and $\Delta(\sigma,R)$ as in \eqref{eq-defin-discriminant-divisor}. 
The main goal of this section if to prove the following.

\begin{proposition}
\label{prop-delta-prime-is-invariant}
Under Assumptions~\ref{setting-equivariant-fibrations},  there exists a divisor $R'$ on $X$ such that~\mbox{$R'\sim_{\mathbb{Q}} R/Y$}, $R'\geq R$, and $\Delta(\sigma,R')$ is a $\overline{\sigma}(\mathrm{Bim}_{\mathrm{reg}}(X;\sigma))$-invariant divisor on~$Y$.
\end{proposition}

We start with recalling the following well known result on bimeromorphic maps between relatively minimal elliptic fibrations from a surface to a curve.

\begin{lemma}
\label{lem-on-surfaces-everything-is-biregular}
Let $\sigma\colon S\to C$ and $\sigma'\colon S'\to C'$ be relatively minimal elliptic fibrations, where $S$, $S'$ are compact complex surfaces, and $C$, $C'$ are (smooth) curves. 
Let~\mbox{$\alpha\colon S\dashrightarrow S'$} be a bimeromorphic map. Assume that there exists an  isomorphism~\mbox{$\beta\colon C\to C'$} such that 
$$
\beta\circ \sigma = \sigma' \circ \alpha.
$$
Then $\alpha$ is an isomorphism. 
\end{lemma}

Under Assumptions~\ref{setting-equivariant-fibrations}, we define the following divisor on $Y$. 
Denote by $\Delta_{\mathrm{sing}}$ the maximal reduced effective divisor on $Y$ such that the preimage via $\sigma$ of a typical point in each component of $\Delta_{\mathrm{sing}}$ is singular. 

\begin{corollary}
\label{cor-delta-is-preserved-on-surface}
Let $\sigma\colon S\to C$ be a relatively minimal elliptic fibration where $S$ is a smooth compact complex surface and $C$ is a curve.   
Let $\Delta=\Delta(\sigma)$ be the discriminant $\mathbb{Q}$-divisor.  
Then 
\begin{enumerate}
\item 
$\mathrm{Bim}(S;\sigma)=\mathrm{Aut}(S;\sigma)$, 
\item 
%$\Delta\leq \Delta'$,
%\item 
$\Delta$ and $\Delta_{\mathrm{sing}}$ are  $\overline{\sigma}(\mathrm{Aut}(S;\sigma))$-invariant. 
\end{enumerate}
\end{corollary}

Example \ref{exam-delta-1} and Example \ref{exam-delta-2} below show that the group $\overline{\alpha}(\mathrm{Bim}_{\mathrm{reg}}(X;\sigma))$ may not preserve $\Delta$ and $\Delta_{\mathrm{sing}}$ if $\sigma$ is not relatively minimal. Furthermore, Example \ref{exam-delta-3} shows that $\Delta(\sigma, R)$ may be not preserved by $\overline{\alpha}(\mathrm{Bim}_{\mathrm{reg}}(X;\sigma))$ even if $\sigma$ is relatively minimal.

\begin{example}
\label{exam-delta-1}
Put 
$\pi\colon X'=C\times\mathbb{P}^1\to\mathbb{P}^1$, 
where $C$ is an elliptic curve and $\pi$ is the projection. Let $\alpha\in\mathrm{Aut}(X')$ be an element which acts trivially on the first factor of $X'=C\times\mathbb{P}^1$, and as an element of order $2$ on the second factor. Hence $\pi$ is $\alpha$-equivariant. Let $P$ be a point on $\mathbb{P}^1$, and let $F=\pi^{-1}(P)$. Blow up a point $Q\in F$ to obtain a morphism $f\colon X\to X'$. It induces a fibration 
$$
\sigma=\pi \circ f\colon X\to \mathbb{P}^1
$$ 
such that $\sigma^{-1}(P)$ is a reducible fiber. 
Hence $\Delta_{\mathrm{sing}}=P$. 
Note that $\alpha$ lifts to an element of $\mathrm{Bim}_{\mathrm{reg}}(X;\sigma)$, but $\overline{\sigma}(\alpha)$ does not preserve $\Delta_{\mathrm{sing}}$.
\end{example}

\begin{example}
\label{exam-delta-2}
Let 
$
\pi\colon X'\to\mathbb{P}^1
$
be a relatively minimal elliptic fibration with only nodal singular fibers. Let $F_1$ and $F_2$ be two such singular fibers. Put~\mbox{$\pi(F_i)=P_i\in \mathbb{P}^1$} for $i=1,2$.   
Suppose that there exists an automorphism $\alpha\in \Aut(X)$ which interchanges $F_1$ and~$F_2$. Blow up the node in $F_1$ to obtain a morphism $f\colon X\to X'$. Note that we have 
$$
f^*F_1= F'_1 + 2E,
$$ 
where $F'_1$ is the proper transform of $F_1$ and $E$ is the $f$-exceptional curve.  For the induced fibration $\sigma\colon X\to \mathbb{P}^1$ we  have that $\Delta(\sigma)=\frac{1}{2} P_1$ which is not $\overline{\sigma}(\alpha)$-invariant. 
\end{example}

\begin{example}
\label{exam-delta-3}
Put 
$
\sigma\colon X=C\times\mathbb{P}^1\to\mathbb{P}^1,
$
where $C$ is an elliptic curve and~$\sigma$ is the projection. 
Let $\alpha\in\mathrm{Aut}(X)$ be an element which acts trivially on the first factor of $X=C\times\mathbb{P}^1$, and as an element of order $2$ on the second factor.  
Put~\mbox{$R=-F$}, where $F$ is a fiber of $\sigma$, and $\sigma(F)=P\in\mathbb{P}^1$. Then $\sigma\colon(X,R)\to \mathbb{P}^1$ is an lc-trivial fibration. Note that $\Delta(\sigma,R)=-P$ is not $\overline{\sigma}(\alpha)$-invariant. 
\end{example}

Now we construct a divisor $R'$ as in Proposition \ref{prop-delta-prime-is-invariant}. 

\begin{construction}
\label{construction-xi}
Under Assumptions~\ref{setting-equivariant-fibrations}, decompose $R$ as follows: 
\[
R=R_1+R_2,
\]
where $\mathrm{codim}_Y\sigma(R_1)=1$ and  $\mathrm{codim}_Y\sigma(R_2)\geq 2$. 
Note that we have
$$
\Delta(\sigma, R)=\Delta(\sigma,R_1).
$$

Let $R'_1$ be a $\QQ$-divisor on $X$ of the form $R_1+\sigma^* D$, 
where $D$ is an effective $\mathbb{Q}$-divisor on $Y$, 
such that $R'_1$ is a maximal anti-effective $\mathbb{Q}$-divisor 
with respect to pull-backs of $\mathbb{Q}$-effective divisors from $Y$. This means that $R'_1$ is anti-effective, and $R'_1+\sigma^* D$ is not anti-effective for any effective $\mathbb{Q}$-divisor $D$ on $Y$. 
Put 
$$
R'=R'_1+R'_2,
$$
where~\mbox{$R'_2=R_2$}. 
Since $R'-R=\sigma^* D$, we have $R'\sim_{\mathbb{Q}} R/Y$ and $R'\geq R$. 
Note that $(X, R')$ is an lc sub-pair since so is $(X, R)$ and $R$ is a simple normal crossing $\mathbb{Q}$-divisor. Hence, $\sigma\colon (X,R')\to Y$ is an lc-trivial fibration. 
Put 
\begin{equation}
\label{eq-defin-xi}
\Xi=\Delta(\sigma, R')=\Delta(\sigma, R'_1).
\end{equation}
Since $R'-R=\sigma^* D$, by Remark \ref{rem-semi-additivity} we have $\Delta(\sigma, R) \leq \Xi$. 
\end{construction}

We are going to show that $\Xi$ is  $\overline{\sigma}(\mathrm{Bim}_{\mathrm{reg}}(X; \sigma))$-invariant. 

\begin{example}
We have $\Xi=0$ in Examples \ref{exam-delta-1}, \ref{exam-delta-2}, and \ref{exam-delta-3}.
\end{example}

The next proposition is not necessary for the proof of the main result. However, we find that it clarifies the structure of the divisor $\Xi$ (see, in particular, Corollary~\ref{not-important-corollary} below). The reader who is interested only in the proof of the main result is advised to skip to Definition~\ref{def-mult-mon}. 

\begin{proposition}
\label{prop-xi-is-effective}
Let $R'=R'_1+R'_2$ be as in Construction \ref{construction-xi}.  
Then 
\begin{enumerate}
\item 
$R'_1$ and $\Xi$ depend only on the fibration $\sigma\colon X\to Y$, 
\item 
$\Xi$ is effective. 
\end{enumerate}
\end{proposition}
\begin{proof} 
We will show that the coefficients of $R'_1$ are uniquely determined in terms of some geometric data associated with the fibration $\sigma\colon X\to Y$. 
Let $C$ be a smooth curve which is an intersection of $\dim X-2$ general very ample divisors on $Y$, and let $S$ be the preimage of $C$ via $\sigma$. Put~\mbox{$\sigma_S=\sigma|_S$}. Then 
$$
\sigma_S\colon S\to C
$$ 
is a (not necessarily relatively minimal) elliptic fibration. 
By adjunction, we have  
$$
K_S+R'_S\sim_{\mathbb{Q}}0/C,
$$
where 
$$
R'_S=R'|_S=R'_1|_S;
$$ 
here the last equality holds since $\mathrm{codim}_Y\sigma(R'_2)\geq 2$. 
Inductively applying Lemma~\ref{lem-delta-restricted}, we see that 
$$
\Delta(\sigma,R')|_C=\Delta(\sigma,R'_1)|_C=\Delta(\sigma_S,R'_S).
$$

Let $\phi\colon S\to T$ be the relatively  minimal model of $S$ over $C$. Thus, there is a relatively minimal elliptic fibration 
\[
\sigma_T\colon T\to C.
\]
We have 
$$
\phi_*(K_S+R'_S)=K_T+R'_T\sim_{\mathbb{Q}} 0/C
$$
and 
\begin{equation}
\label{eq-S-T}
K_S+R'_S=\phi^*(K_T+R'_T).
\end{equation}  
It follows that  $\Delta(\sigma_T,R'_T)=\Delta(\sigma_S,R'_S)$. 

Since $\sigma_T\colon T\to C$ is relatively minimal, we have $K_T\sim_{\mathbb{Q}}0/C$. 
Hence, $R'_T\sim_{\mathbb{Q}}0/C$. 
We claim that $R'_T=0$. 
Indeed, suppose that $R'_T\neq 0$. 
We claim that under this assumption $-R'_1\geq \sigma^*D$ for some non-zero effective divisor $D$ on $Y$. 
Since $R'_1$ is anti-effective, $R'_S$ and $R'_T$ are anti-effective as well. Then $-R'_T\geq \sigma_T^* D_C$ for some non-zero effective $\mathbb{Q}$-divisor $D_C$ on~$C$. 
It follows from~\eqref{eq-S-T} that 
\begin{equation}
\label{eq-last-inequality}
-R'_S\geq \sigma_S^* D_C.
\end{equation}
Write $R'_1=\sum a_i D_i$, where $D_i$ are $\sigma$-vertical prime divisors on $X$, $a_i\in \mathbb{Q}$, and~$a_i\leq 0$. We know from \eqref{eq-last-inequality} that 
\begin{equation}
\label{eq-last-equality}
-R'_S = -R'_1|_S\geq \sigma^* D_C.
\end{equation}
Let $R''_1=\sum a_i D_i$ be the sum of the components of $R'_1$ such that $\sigma(D_i|_S)$ is contained in the support of $D_C$.  From~\eqref{eq-last-equality} we get  $-R''_1\geq \epsilon \sigma^*(\sigma(R''_1))$ for some $0<\epsilon\ll 1$. 
Then
\[
-R'_1\geq -R''_1\geq \epsilon \sigma^*(\sigma(R''_1))> 0.
\]
A contradiction with the choice of $R'_1$ shows that $R'_T=0$. 

Thus, the divisor $R'_S\sim_{\mathbb{Q}}-K_S$ is defined uniquely by the formula 
\[
\phi^*K_T=K_S+R'_S. 
\]
Since $R'_1|_S=R'_S$, it follows that $R'_1$ is uniquely determined by the minimal model of~$S$, and hence $R'_1$ depends only on the fibration $\sigma\colon X\to Y$. 
Furthermore, since 
$$
\Xi=\Delta(\sigma, R'_1),
$$
we see that $\Xi$ depends only on the fibration $\sigma\colon X\to Y$.  

Finally, we show that $\Xi$ is effective. Observe that $\Delta(\sigma_T,R'_T)=\Delta(\sigma_T,0)$ is effective, and hence 
$$
\Xi|_C=\Delta(\sigma,R'_1)|_C=\Delta(\sigma_S, R'_S)=\Delta(\sigma_T,R'_T)
$$ 
is effective.
Therefore, $\Xi$ is effective as well. 
\end{proof}

\begin{corollary}
\label{prop-delta-leq-xi}
Under Assumptions~\ref{setting-equivariant-fibrations}, 
suppose that $R=R_1+R_2$, where 
$\mathrm{codim}_Y\sigma(R_1)=1$,   $\mathrm{codim}_Y\sigma(R_2)\geq 2$, 
and $R_1$ is maximal anti-effective $\mathbb{Q}$-divisor on~$X$ with respect to pull-back of effective divisors from $Y$. Then $\Delta(\sigma,R)=\Xi$.
\end{corollary}

Combining Construction \ref{construction-xi} and Corollary \ref{prop-delta-leq-xi}, we obtain 

\begin{corollary}
\label{not-important-corollary}
Under Assumptions~\ref{setting-equivariant-fibrations}, we have 
\[
\Delta(\sigma, R)\leq \Xi =\Delta(\sigma,R')\leq \Delta(\sigma)\leq \Delta_{\mathrm{sing}}.
\]
\end{corollary}

We show that the divisor $\Xi$ admits an alternative characterization. In particular, it will give another proof of Proposition \ref{prop-xi-is-effective}.

\begin{definition}
\label{def-mult-mon}
Under Assumptions~\ref{setting-equivariant-fibrations}, let $\Delta_{\mathrm{mult}, m}$ be the maximal effective reduced divisor on $Y$ such that the preimage via $\sigma$ of each component of $\Delta_{\mathrm{mult},m}$ has multiplicity $m\geq 2$
(i.e., it is a divisor such that the greatest common divisor of the coefficients at its 
irreducible components equals~\mbox{$m\geq 2$}). 
Put 
\[
\Delta_{\mathrm{mult}}=\sum_{m\geq 2}(1-1/m)\Delta_{\mathrm{mult},m}.
\]
Let $\Delta_{\mathrm{mon}} = \sum a_i D_i$, where $D_i$ are all the prime divisors on $Y$ such that the monodromy around $D_i$ is as in the types $I,II,III,IV,I^*,II^*,III^*,IV^*$ in Table~\ref{table:monodromy} from Section~\ref{sec-canonical-bundle}, and the coefficients $a_i$ are equal to $1-\mathrm{lct}(S,F)$ for the corresponding type.  
\end{definition}

\begin{proposition}
\label{prop-monodromy-description}
Under Assumptions~\ref{setting-equivariant-fibrations}, we have 
\[
\Xi = \Delta_{\mathrm{mult}} + \Delta_{\mathrm{mon}}. 
\]
\end{proposition}

\begin{proof}
Consider the following properties of a prime divisor $D$ on $Y$:
\begin{enumerate}
\item
$\sigma^*D$ is not multiple,  
and the monodromy of $\sigma$ around $D$ is not unipotent (and in particular, it is non-trivial), or
\item
$\sigma^*D$ has multiplicity $m\geq 2$. 
\end{enumerate}
Note that the conditions (1) and (2) do not change under taking general hyperplane section $H$ on $Y$ and replacing $\sigma$ by the fibration $\sigma|_{V}\colon V\to H$, where $V=\sigma^{-1}(H)$.
Furthermore, if $\dim X=2$, then conditions (1) and (2) are invariant under running a relative minimal program on the compact complex surface $X$ over the curve $Y$. 
Thus it is enough to assume that $\sigma\colon X\to Y$ is a relatively minimal elliptic fibration over a curve. In this case, if $\sigma^*D$ is not multiple, 
we can compute the coefficient at~$D$ in $\Xi$ in terms of the monodromy around $D$ by Proposition \ref{prop-monodromy-determines-everything}.  
If $\sigma^*D$ is multiple and has multiplicity $m$, then the coefficient at $D$ in~$\Xi$ is equal to~\mbox{$(1-1/m)$}, see Table \ref{table:monodromy}. 
\end{proof}

Using Proposition \ref{prop-monodromy-description}, we will show that $\Xi$ is  $\overline{\sigma}(\mathrm{Bim}_{\mathrm{reg}}(X;\sigma))$-invariant. 
To do this, it is enough to show that both $\Delta_{\mult}$ and $\Delta_{\mathrm{mon}}$ are $\overline{\sigma}(\mathrm{Bim}_{\mathrm{reg}}(X;\sigma))$-invariant.

\begin{proposition}
\label{prop-mult-is-preserved}
Under Assumptions~\ref{setting-equivariant-fibrations}, the $\QQ$-divisors $\Delta_{\mathrm{mult},m}$ 
and $\Delta_{\mathrm{mult}}$ 
are $\overline{\sigma}(\mathrm{Bim}_{\mathrm{reg}}(X;\sigma))$-invariant. 
\end{proposition}

\begin{proof}
It is enough to prove the assertion for the divisor $\Delta_{\mathrm{mult},m}$.

Take $\alpha\in\mathrm{Bim}_{\mathrm{reg}}(X;\sigma)$. Let $\overline{\alpha}$ be the image of $\alpha$ in 
$$
\overline{\sigma}(\mathrm{Bim}_{\mathrm{reg}}(X;\sigma))\subset \mathrm{Aut}(Y).
$$  
Resolve the indeterminacy of $\alpha$. We obtain a diagram
\begin{equation*}
\begin{tikzcd}
 & W \arrow[dl, "\phi"'] \arrow{dr}{\psi} &  \\
X \arrow{d}{\sigma} \arrow[dashed]{rr}{\alpha}  &  & X \arrow{d}{\sigma} \\
Y \arrow{rr}{\overline{\alpha}} & & Y
\end{tikzcd}
\end{equation*} 
where $\phi$  and $\psi$ are bimeromorphic contractions. Let $\overline{D}$ be an irreducible component of $\Delta_{\mathrm{mult}}$. This means that 
$$
D=\sigma^*(\overline{D})=\sum d_i D_i
$$ 
with the greatest common divisor of the coefficients $d_i$ equal to $m\geq 2$. Then we have $D=mD'$ for some divisor $D'$. Thus, one has 
\begin{equation}\label{eq:multiplicity-after-automorphism}
\psi_*\phi^*D=m\psi_*\phi^*D'=\sigma^*(\overline{\alpha}(\overline{D})),
\end{equation} 
so $\overline{\alpha}(\overline{D})$ is an irreducible component of $\Delta_{\mathrm{mult}}$. 
Furthermore, if we write 
$$
\sigma^*(\overline{\alpha}(\overline{D}))=\sum d'_i D'_i,
$$ 
and denote by $m'$ the greatest common divisor of the coefficients $d_i'$, 
then~\eqref{eq:multiplicity-after-automorphism} shows that $m$ divides $m'$.   
Applying a similar argument to $\alpha^{-1}$, we see that $m'$ divides $m$, so that $m=m'$. 
This shows that $\Delta_{\mathrm{mult},m}$ is $\overline{\alpha}$-invariant.
\end{proof}

\begin{remark}
One can generalize Proposition \ref{prop-mult-is-preserved} to the case of arbitrary (not necessarily elliptic) fibrations from a smooth compact complex manifold. 
On the other hand, one cannot get rid of the smoothness assumption even for elliptic surfaces. 
\end{remark}

\begin{proposition}\label{prop-mon-is-preserved} 
Under Assumptions~\ref{setting-equivariant-fibrations}, the divisor $\Delta_{\mathrm{mon}}$ is $\overline{\sigma}(\mathrm{Bim}_{\mathrm{reg}}(X;\sigma))$-invariant.  
\end{proposition}
\begin{proof}
Let $\alpha$ be an element in $\mathrm{Bim}_{\mathrm{reg}}(X;\sigma)$. Denote by $I$ the union of the indeterminacy loci of $\alpha$ and $\alpha^{-1}$ on $X$, and all fibers of $\sigma$ which are not elliptic curves. 
Set $\overline{I}=\sigma(I)$. Thus, $\overline{I}$ is a closed subset of codimension at least $1$ in~$Y$. 
Let $D$ be an irreducible component of $\Delta_{\mathrm{mon}}$. Let $\overline{\gamma}$ be a small loop going around $D$ in the counter-clockwise direction (with respect to the natural orientation on~$Y$), which has no common points with $\overline{I}$. Set $\gamma=\sigma^{-1}(\overline{\gamma})$. Then 
\[
\sigma|_{\gamma}\colon \gamma\to\overline{\gamma}
\] 
is a differentially locally trivial fibration into elliptic curves over a circle. 
We see that~$\alpha$ induces a homeomorphism between $\gamma$ and $\alpha(\gamma)$. It follows that the monodromy around $D$ and $\overline{\alpha}(D)$ is the same. \end{proof}

Combining Propositions~\ref{prop-mult-is-preserved} and~\ref{prop-mon-is-preserved}, we obtain

\begin{corollary}
Under Assumptions~\ref{setting-equivariant-fibrations}, the divisor 
$$
\Xi = \Delta_{\mathrm{mult}} + \Delta_{\mathrm{mon}}
$$ 
is $\overline{\sigma}(\mathrm{Bim}_{\mathrm{reg}}(X;\sigma))$-invariant.
\end{corollary}

We give another proof that $\Xi$ is $\overline{\sigma}(\mathrm{Bim}_{\mathrm{reg}}(X;\sigma))$-invariant, which could be useful for generalizations. 

\begin{proposition}
Under Assumptions~\ref{setting-equivariant-fibrations}, the divisor $\Xi$ is $\overline{\sigma}(\mathrm{Bim}_{\mathrm{reg}}(X;\sigma))$-invariant. 
\end{proposition}
\begin{proof}
Take $\alpha\in\mathrm{Bim}_{\mathrm{reg}}(X;\sigma)$. Denote by $\overline{\alpha}$ the image of $\alpha$ in
$$
\overline{\sigma}(\mathrm{Bim}_{\mathrm{reg}}(X;\sigma))\subset\mathrm{Aut}(Y).
$$ 
Taking successive hyperplane sections of $Y$ and passing to their preimages via $\sigma$ we obtain two fibrations $S\to C$ and $S'\to C'$ whose typical fibers are elliptic curves, where $S$ and $S'$ are smooth comp    act complex surfaces, $C$ and $C'$ are smooth projective curves with $\overline{\alpha}(C)=C'$. Also, we have a bimeromorphic map 
\[
\alpha_S=\alpha|_S\colon S\dashrightarrow S'.
\]
We also have $\mathbb{Q}$-divisors $R_S$ on $S$ and $R_{S'}$ on $S'$ such that $K_{S}+R_S\sim_{\mathbb{Q}}0/C$ and~\mbox{$K_{S'}+R_{S'}\sim_{\mathbb{Q}}0/C'$}, and both 
$-R_S$ and $-R_{S'}$ are effective. 

Consider relatively minimal models $S\to T$ over $C$ and $S'\to T'$ over $C'$, and let 
$$
K_T+R_T\sim_{\mathbb{Q}}0/C, \qquad K_{T'}+R_{T'}\sim_{\mathbb{Q}}0/C'
$$
be the push-forwards of $K_S+R_S$ and $K_{S'}+R_{S'}$, respectively. 
We obtain the following commutative diagram
\begin{equation*}
\begin{tikzcd}
(S,R_S) \arrow{d} \arrow[dashed]{rr}{\alpha_S} &  & (S', R_{S'}) \arrow{d} \\
(T,R_T) \arrow{d}{\sigma_T} \arrow{rr}{\alpha_T}  &  & (T',R_{T'}) \arrow{d}{\sigma_{T'}} \\
C \arrow{rr}{\overline{\alpha}_C} & & C'
\end{tikzcd}
\end{equation*}
Here by Lemma \ref{lem-on-surfaces-everything-is-biregular} the map $\alpha_T$ is an isomorphism between $T$ and~$T'$. 
It follows from Corollary~\ref{cor-delta-is-preserved-on-surface} that $\overline{\alpha}_C(\Delta(\sigma_T))=\Delta(\sigma_{T'})$.  
Inductively applying Lemma~\ref{lem-delta-restricted}, we get 
\[
\Xi|_C = \Delta(\sigma_S,R_S) = \Delta(\sigma_T), \quad \quad \quad \quad \Xi|_{C'} = \Delta(\sigma_S',R_{S'}) = \Delta(\sigma_{T'}). 
\]
Therefore, 
\[
\overline{\alpha}(\Xi)|_{C'}=\overline{\alpha}_C(\Xi|_{C})=\overline{\alpha}_C(\Delta(\sigma_T)) = \Delta(\sigma_{T'}) = \Xi|_{C'}.
\]
%By Remark \ref{rem-restriction-of-two-divisors}, 
It follows that $\Xi=\overline{\alpha}(\Xi)$, 
which gives the required assertion. 
\end{proof}

We conclude this section by an observation concerning Stein factorizations 
of certain morphisms.

\begin{lemma}
\label{lem-existence-of-lifting}
Let $\sigma\colon X\to Y$ be a fibration from a compact complex manifold~$X$ to a projective variety $Y$.  Let $G\subset \mathrm{Bim}_{\mathrm{reg}}(X;\sigma)$  
be a subgroup. Denote its image by~\mbox{$\overline{G}\subset \mathrm{Aut}(Y)$}. Let 
\begin{equation}
\label{eq-stein-factorization}
X\to X_1\to Y
\end{equation}
be the Stein factorization of $\sigma$. Then $\overline{G}$ admits a lifting $\overline{G}_1\subset \mathrm{Aut}(X_1)$ such that~\mbox{$\overline{G}_1\cong \overline{G}$}, and  factorization~\eqref{eq-stein-factorization} is $G$-equivariant. 
\end{lemma}

\begin{proof}
We argue that the claim follows from the universal property of the Stein factorization, cf. \cite[10, §6.1]{GR84}. 
Let $\alpha\in G$. 
We shall construct a uniquely defined lifting $\overline{\alpha}_1\in \mathrm{Aut}(X_1)$ of $\overline{\alpha}\in \mathrm{Aut}(X)$ which fits into the following commutative diagram:
\begin{equation*}
\begin{tikzcd}
 & W \arrow[dl, "\phi"'] \arrow{dr}{\psi} &  \\
X \arrow{d}{\sigma_1} \arrow[dashed]{rr}{\alpha}  &  & X \arrow{d}{\sigma_1} \\
X_1 \arrow{d}{\sigma} \arrow{rr}{\overline{\alpha}_1}  &  & X_1 \arrow{d}{\sigma} \\
Y \arrow{rr}{\overline{\alpha}} & & Y
\end{tikzcd}
\end{equation*}
Here $X\xleftarrow{\phi}W \xrightarrow{\psi}X$ is a resolution 
of the indeterminacy of the map $\alpha$. 
Note that the map $\sigma_1\circ\psi\colon W\to X_1$ is constant on the fibers of the composition $\sigma \circ\sigma_1\circ \phi$. Indeed, this follows from the commutativity of the diagram. 
Note also that the fibers of~\mbox{$\sigma_1\circ \phi$} and $\sigma_1\circ \psi$ are connected. In fact, 
$$
W\xrightarrow{\sigma_1\circ \phi} X_1\xrightarrow{\phantom{\phi}\sigma\phantom{\phi}} Y
$$ 
is the Stein factorization of $W\xrightarrow{\sigma \circ \sigma_1\circ \phi} Y$, and 
$$
W\xrightarrow{\sigma_1\circ \psi} X_1\xrightarrow{\phantom{\phi}\sigma\phantom{\phi}} Y
$$ 
is the Stein factorization of $W\xrightarrow{\sigma \circ \sigma_1\circ \psi} Y$. 
Hence, by the universal property of the Stein factorization there exists a unique morphism $\overline{\alpha_1}\colon X_1\to X_1$ such that 
$$
\sigma_1 \circ \psi = \overline{\alpha}_1\circ \sigma_1 \circ \phi.
$$
Since the lifting $\overline{\alpha}_1\in \mathrm{Aut}(X_1)$ of $\overline{\alpha}\in \mathrm{Aut}(X)$ is unique, it follows that the group~$\overline{G}$ admits a lifting $\overline{G}_1\subset \mathrm{Aut}(X_1)$ such that $\overline{G}_1\cong\overline{G}$ and factorization~\eqref{eq-stein-factorization} is $G$-equivariant. 
\end{proof}

\section{Bimeromorphic modifications}
Let $X$ be a compact complex manifold of dimension $n\geq 2$ such that for its Kodaira dimension we have $\kappa(X)= n-1$. Our construction is similar to that of~\mbox{\cite[Theorem 6.11]{Ue75}}. 
We fix the notation.  
Let $m$ be a number such that the linear system $|mK_X|$ defines the pluricanonical map $\sigma\colon X\dashrightarrow Y$. In this section, we construct certain bimeromorphic modifications of $X$ and $Y$ and prove the main results of the paper. 
We denote by $$
\Gamma=\overline{\rho}(\mathrm{Bim}(X))\subset \mathrm{PGL}(\mathrm{H}^0(X, \oo_X(mK_X)))
$$ 
the image of the projective pluricanonical representation.    

\begin{proposition}
\label{prop-first-modification}
Let $X$ be a compact complex manifold, and let $\phi\colon X\dashrightarrow Y$ be the pluricanonical map. Then there exists the following $\mathrm{Bim}(X)$-equivariant commutative diagram
\begin{equation}
\label{diag-first-mod}
\begin{tikzcd}
X_1 \arrow{d}{\sigma_1} \arrow{r}{\phi}  & X \arrow[dashed]{d}{\sigma} \\
Y_1 \arrow{r}{\psi} & Y  
\end{tikzcd}
\end{equation}
such that 
\begin{enumerate}
\item 
\label{first-mod-1}
$X_1$ and $Y_1$ are compact complex manifolds (in particular, they are smooth),
\item
\label{first-mod-2}
$\phi$ and $\psi$ are bimeromorphic modifications, 
\item 
\label{first-mod-3}
$\Gamma$ acts on $Y_1$ biregularly and 
$\psi$ is $\Gamma$-equivariant, 
\item
\label{first-mod-4}
$\sigma_1$ is a fibration and 
\[
K_{X_1} + R_1 \sim_{N_1} \sigma_1^*H_1,
\]
where $H_1$ is a big and nef $\mathbb{Q}$-divisor on $Y_1$, and $-R_1$ is an effective integral divisor on~$X_1$,
\item 
\label{first-mod-5}
the sub-pair $(X_1, R_1)$ is lc.
\end{enumerate}
\end{proposition}
\begin{proof}
%First of all, passing to the Stein factorization if needed, we may assume that $Y$ is normal. 
%Since $\dim Y=\dim X - 1$ and $X$ is smooth, the indeterminacy locus of the map $\sigma\colon X\dashrightarrow Y$ does not dominate $Y$. Also, since $X$ is smooth, a typical fiber of $\sigma$ is a smooth curve. 
%X'_1=\widetilde{Y_1 \times_{Y_0} X_0}

We construct the following commutative diagram:
\begin{equation*}
\begin{tikzcd}
X_1   \arrow{r}{\phi_1} \arrow{d}{\sigma_1}  & X_0 \arrow{d}{\sigma_0} \arrow{r}{\phi_0} & X \arrow[dashed]{d}{\sigma} \\
Y_1 \arrow{r}{\psi_1} & Y_0 \arrow{r}{\psi_0} & Y  
\end{tikzcd}
\end{equation*}
where 
\begin{itemize}
\item
$\phi_0$ is a
resolution of the indeterminacy of the map $\sigma$,
 \item
$\sigma_0$ is the Stein factorization of $\sigma \circ \phi_0$, so the fibers of $\sigma_0$ are connected and $Y_0$ is normal, %We have $K_X = \sigma^*H$ where $H$ is an ample $\mathbb{Q}$-divisor on $Y$. 
\iffalse
Hence, we have 
\[
K_{X_0}=\phi_0^*K_X + E_0 = \phi_0^*\sigma^*H + E_0 = \sigma_0^* H_0 + E_0
\]
where $E_0$ is a $\psi_0$-exceptional $\mathbb{Q}$-divisor on $X_0$, and $H_0=\psi_0^*(H)$ is a big and nef $\mathbb{Q}$-divisor on $Y_0$. 
\fi
\item
$\psi_1$
is a resolution of singularities of $Y_0$, so $Y_1$ is smooth,
\item
$X_1$ is a resolution of singularities of the main component of the fiber product ${X_0 \times_{Y_0} Y_1}$, and $\phi_1\colon X_1\to X_0$ and $\sigma_1\colon X_1\to Y_1$ are the induced maps.
%$\widetilde{Y_1 \times_{Y_0} X_0}$ is the normalization , 
\iffalse
We obtain 
\[
K_{X'_1} = (\phi'_1)^*K_{X_0} + F_0 = (\phi'_1)^*(\sigma_0^*H_0 + E_0) + F_0 = (\phi'_1)^*\sigma_0^* H_0 + F'_0 = (\sigma'_1)^*H_1 + F'_0
\]
where $F_0$ is a $\phi'_1$-exceptional $\mathbb{Q}$-divisor on $X'_1$, $F'_0$ is a $\mathbb{Q}$-divisor on $X'_1$, and $H_1=\psi^*H_0$ is a big and nef $\mathbb{Q}$-divisor on $Y_1$.  
\fi 
\end{itemize}
\iffalse
We obtain the following diagram:
\begin{equation*}
\begin{tikzcd}
X_1  \arrow{r}{\phi'_1} \arrow{dr}{\sigma_1}  & X'_1 \arrow{d}{\sigma'_1} \\
& Y_1  
\end{tikzcd}
\end{equation*}
\fi
Put $\phi=\phi_0\circ \phi_1$ and $\psi=\psi_0\circ\psi_1$. 
By construction $X_1$ and $Y_1$ are smooth and compact, which proves (\ref{first-mod-1}). 
Also, by construction $\phi$ is a bimeromorphic modification, and $\psi$ is generically finite. 
On the other hand, we know from Remark \ref{rem-ueno-connected} that the closure of a typical fiber of $\sigma$ is connected. Hence, $\psi_0$ is a bimeromorphic modification as well. This proves (\ref{first-mod-2}). 
By Lemma \ref{lem-existence-of-lifting}, the group $\Gamma$ admits a lifting to $\Gamma\subset \mathrm{Aut}(Y_0)$. By~\mbox{\cite[Theorem 13.2]{BM}}, a resolution of singularities $\psi_1$ can be chosen $\Gamma$-equivariant. This proves (\ref{first-mod-3}). 
Since $\phi$ is a bimeromorphic modification, it is $\mathrm{Bim}(X)$-equivariant. 
By construction, $\sigma_1$ is $\mathrm{Bim}(X)$-equivariant as well, so the diagram \eqref{diag-first-mod} is $\mathrm{Bim}(X)$-equivariant. 

Note that $mK_X\sim \sigma^*L$ where $L$ is a very ample divisor on~$Y$. 
Here by the pullback of a very ample divisor via a meromorphic map we mean the proper transform of a general divisor linearly equivalent to it.  
%Hence $L_0=\psi_0^*L$ is a big and nef divisor on $Y_0$. 
Put $H=1/mL$, so $H$ is an ample $\mathbb{Q}$-divisor on $Y$. 
We have $K_X\sim_m \sigma^*H$. 
Recall that $X$ is smooth. Thus, 
\[
K_{X_1} = \phi^*K_{X}+E_1,
\]
where $E_1$ is a $\phi$-exceptional effective  integral divisor on~$X_1$. 
Therefore
\begin{equation}
\label{eq-first-modification_}
K_{X_1} = \phi^*K_{X}+E_1 \sim_m \phi^*\sigma^*H + E_1 \sim_m \phi^*\sigma_1^*H + E_1 \sim_m \sigma_1^*H_1 + E_1
\end{equation}
where $H_1=\phi^*H$ is a big and nef $\mathbb{Q}$-divisor on $Y_1$ such that $mH_1$ is integral. 
We may assume that (possibly after a further blow-up of $X_1$) the divisor $E_1$ has simple normal crossings. Put $R_1=-E_1$ and $N_1=m$. This proves (\ref{first-mod-4}). 

Finally, the sub-pair $(X_1, R_1)$ is lc since $R_1$ has simple normal crossings and $-R_1$ is effective. This proves (\ref{first-mod-5}), and completes the proof of the proposition.   
\end{proof}

\begin{proposition}
\label{prop-elliptic-or-moishezon}
In the notation of Proposition \ref{prop-first-modification}, assume that for the Kodaira dimension of $X$ we have $\kappa(X)=\dim X-1$. 
Then either $X$ is Moishezon, or a typical fiber of $\sigma_1\colon X_1\to Y_1$ is an elliptic curve, and $R_1$ is $\sigma_1$-vertical.
\end{proposition}
\begin{proof}
Let $F$ be a typical fiber of $\sigma_1\colon X_1\to Y_1$, so $F$ is a smooth curve. 
Since for the Kodaira dimension of $X_1$ we have $\kappa(X_1)=\kappa(X)\geq 0$, by Lemma~\ref{lem-fiber-non-negative-kodaira-dim} we obtain $\kappa(F)\geq 0$. On the other hand, if $\kappa(F)> 0$, we show that $X_1$, and hence~$X$, is Moishezon. Indeed, assume that $\kappa(F)> 0$. 
By~\mbox{\cite[Theorem V.4.10]{BS76}}, there exists a Zariski open subset $U\subset Y_1$
such that the morphism 
$$
\sigma_1|_{V}\colon V\to U
$$ 
is smooth, where~\mbox{$V=\sigma_1^{-1}(U)$}.  
Then $K_{X_1}$ restricts to an ample line bundle on any fiber over a point in $U$. Thus, by \cite[Lemma 19, Example 12]{Kollar22} it follows that~$X_1$ is Moishezon. Hence~$X$ is Moishezon as well.  
Similarly, if $R_1$ is not $\sigma_1$-vertical,  then~\mbox{$-R_1$} has a component which dominates $Y_1$. Such a component restricts to an ample line bundle on any fiber over a point in open subset $U$ in $Y$. This implies that~$X_1$ and~$X$ are Moishezon. We conclude that if $X$ is not Moishezon, then~\mbox{$\kappa(F)=0$} (so~$F$ is an elliptic curve) and $R_1$ is $\sigma_1$-vertical. 
\end{proof}

By Proposition \ref{prop-elliptic-or-moishezon}, if $X$ is not Moishezon, then a typical fiber of $\sigma$ is an elliptic curve. In this case, we can consider the $j$-invariant map $j\colon Y\dashrightarrow \mathbb{P}^1$. 

\begin{lemma}
\label{prop-action-on-target-of-j-map}
In the notation of Proposition \ref{prop-first-modification}, assume that $\kappa(X)=\dim X-1$ and  the typical fiber of $\sigma_1\colon X_1\to Y_1$ is an elliptic curve. Then $j\colon Y_1\dashrightarrow \mathbb{P}^1$ is $\Gamma$-equivariant, 
and $\Gamma$ acts on  $\mathbb{P}^1$ trivially. 
\end{lemma}
\begin{proof}
Indeed, since the morphism $\sigma_1\colon X_1\to Y_1$ is 
$\mathrm{Bim}(X)$-equivariant,  
every element of $\mathrm{Bim}(X)$ maps an elliptic curve which is a typical fiber of $\sigma_1$ to an isomorphic elliptic curve; hence these elliptic curves have the same values of the $j$-invariant. This shows that the map~\mbox{$j\colon Y_1\dashrightarrow \mathbb{P}^1$} is $\Gamma$-equivariant with respect to  the trivial action of $\Gamma$ on $\mathbb{P}^1$.
\end{proof}

\begin{proposition}
\label{prop-second-modification} 
In the notation of Proposition \ref{prop-first-modification}, assume that $\kappa(X)=\dim X-1$ and  the typical fiber of $\sigma_1\colon X_1\to Y_1$ is an elliptic curve. 
Then there exists the following $\mathrm{Bim}(X)$-equivariant commutative diagram
\begin{equation}
\label{eq-diagram-second-modification}
\begin{tikzcd}
X_2 \arrow{d}{\sigma_2} \arrow{r}{\phi_2}  & X_1 \arrow{d}{\sigma_1} \\
Y_2 \arrow{r}{\psi_2} \arrow{d}{j}  & Y_1 \arrow[dashed]{d}{j} \\
\mathbb{P}^1 \arrow[equal]{r} & \mathbb{P}^1
\end{tikzcd}
\end{equation}
such that 
\begin{enumerate}
\item 
\label{second-mod-1}
$X_2$ and $Y_2$ are compact complex manifolds (in particular, they are smooth),
\item
\label{second-mod-2}
$\phi_2$ and $\psi_2$ are bimeromorphic modifications, 
\item 
\label{second-mod-3}
$\Gamma$ acts on $Y_2$ biregularly, and 
$\psi_2$ is $\Gamma$-equivariant, % where $\Gamma_0$ is an extension of $\Gamma$ by a finite group, and $\Gamma_0$ acts on $Y_2$ biholomorphically, 
\item
\label{second-mod-4}
$\sigma_2$ is a contraction, and 
\begin{equation*}
%\label{eq-second-modification}
K_{X_2} + R_2 \sim_{N_2} \sigma_2^*H_2,
\end{equation*}
where $H_2$ is a big and nef divisor on $Y_2$, and $-R_2$ is an effective $\sigma_2$-vertical integral divisor on $X_1$,
%\item 
%\label{second-mod-5}
%the sub-pair $(X_2, R_2)$ is lc,
\item 
\label{second-mod-6}
the map $j\colon Y_2\to\mathbb{P}^1$ coincides with the $j$-invariant map on the typical fiber of $\sigma_2$,
\item
\label{second-mod-7}
the sub-pair $(Y_2, \Delta_{Y_2}+M_{Y_2})$ is lc and $\Gamma$-invariant, where 
$\Delta_{Y_2}=\Delta(\sigma_2,R_2)$ is the discriminant $\mathbb{Q}$-divisor and $M_{Y_2}$ is the moduli $\mathbb{Q}$-divisor as in \eqref{eq-defin-discriminant-and-moduli-divisor},
\item
\label{second-mod-8}
the divisor $K_{Y_2}+\Delta_{Y_2}+M_{Y_2}$ is big and nef.   
\end{enumerate} 
As a consequence, the group $\mathrm{Aut}(Y_2,\Delta_{Y_2}+M_{Y_2})$ is finite, and so its subgroup $\Gamma$ is finite as well. 
\end{proposition}
\begin{proof}
Let us construct the diagram \eqref{eq-diagram-second-modification}. 
By Lemma \ref{prop-action-on-target-of-j-map}, the map $j\colon Y_1\dashrightarrow \mathbb{P}^1$ is $\Gamma$-equivariant where the action of $\Gamma$ on $\mathbb{P}^1$ is trivial. 
Let $\psi_2\colon Y_2\to Y_1$ be a resolution of the indeterminacy of $j\colon Y_1\dashrightarrow \mathbb{P}^1$.
Since $j$ is $\Gamma$-invariant, it follows that $\psi_2$ can be chosen to be $\Gamma$-equivariant. This establishes \eqref{second-mod-6}. 
By further blowing-up, we may assume that $Y_2$ is smooth. 

Let $X_2$ be the resolution of singularities of the main component of $X_1\times_{Y_1} Y_2$, and let $\sigma_2\colon X_2\to Y_2$ and $\phi_2\colon X_2\to X_1$ be the induced maps. We see that $X_2$ and~$Y_2$ are smooth compact complex manifolds, which proves \eqref{second-mod-1}. 
By construciton, $\phi_2$ and~$\psi_2$ are bimeromorphic modifications, which proves \eqref{second-mod-2}.
Note that $\sigma'_2$ is a fibration. 
Similarly to \eqref{eq-first-modification_}, we have 
\begin{equation}
\label{eq-second-modification_}
K_{X_2} + R_2 \sim_{\mathbb{Q}} \sigma_2^*H_2, 
\end{equation}
where $-R_2$ is an effective integral divisor on $X_2$, and $H_2=\psi_2^*H_1$ is a big and nef $\mathbb{Q}$-divisor on $Y_2$. 
By further blowing-up of $X_2$ we may assume that $R_2$ is a simple normal crossing $\mathbb{Q}$-divisor. 
This establishes \eqref{second-mod-4}. 

As in Construction \ref{construction-xi}, by replacing $R_2$ with $R_2+\sigma_2^*D$ for an effective $\mathbb{Q}$-divisor on $Y_2$, we may assume that $R_2$ is maximal anti-effective $\mathbb{Q}$-divisor on $X$ with respect to pull-backs from $Y$. This means that $R_2$ is anti-effective, and $R_2+\sigma_1^* D$ is not anti-effective for any effective $\mathbb{Q}$-divisor $D$ on $Y$. Note that $R_2$ is still a simple normal crossing $\mathbb{Q}$-divisor, and in particular, $(X_2, R_2)$ is an lc sub-pair.  
By Proposition \ref{prop:canonical-bundle-formula} we have  
\begin{equation}
\label{eq-second-mod-canonical-bundle-formula-1}
K_{X_2}+R_2\sim_{\mathbb{Q}} \sigma_2^*\left(K_{Y_2} +  \Delta_{Y_2} + M_{Y_2}\right),
\end{equation}
where $\Delta_{Y_2}=\Delta(Y_2, R_2;\sigma_2)$ is the discriminant $\mathbb{Q}$-divisor of $\sigma_2$, and 
the moduli $\mathbb{Q}$-divisor $M_{Y_2}$ is a nef $\mathbb{Q}$-divisor on $Y_2$ such that $M_{Y_2}\sim_{\mathbb{Q}} \frac{1}{12}j^*Q$ for a point $Q$ in $\mathbb{P}^1$. 
Moreover, the sub-pair $(Y_2,\Delta_{Y_2})$ is lc. 
We may assume that the point $Q\in \mathbb{P}^1$ is general. Hence the sub-pair $(Y_2, \Delta_{Y_2}+M_{Y_2})$ is lc. 
From \eqref{eq-second-modification_} and \eqref{eq-second-mod-canonical-bundle-formula-1} it follows that $K_{Y_2}+\Delta_{Y_2}+M_{Y_2}$ is big and nef, which establishes \eqref{second-mod-8}. 
 
Consider the group $\Gamma'=\mathrm{Aut}(Y_2,\Delta_{Y_2}+M_{Y_2})$. 
By Proposition \ref{prop-group-preserves-pair} we know that~$\Gamma'$ is finite. 
By Proposition~\ref{prop-delta-prime-is-invariant} and Lemma~\ref{prop-action-on-target-of-j-map} the $\QQ$-divisors $\Delta_{Y_2}$ and $M_{Y_2}$ are $\Gamma$-invariant, so $\Gamma\subset \Gamma'$. Hence $\Gamma$ is finite as well. 
This completes the proof of the proposition. 
\end{proof}

Finally, we are ready to prove the main result of the paper.

\begin{proof}[Proof of Theorem \ref{main-theorem}]
This is the last assertion of Proposition \ref{prop-second-modification}.
\end{proof}

\end{document}